\documentclass[12pt,reqno]{amsart}

\usepackage{amsfonts}
\usepackage{amsmath,amssymb, amsthm, mathrsfs}
\usepackage{wasysym}
\usepackage{eufrak}
\usepackage{enumerate}
\usepackage{graphicx}
\usepackage{lineno}
\usepackage{amstext}
\usepackage{amssymb}
\usepackage{amsthm}
\usepackage{amsmath}
\usepackage[all]{xy}
\usepackage{graphicx}
\usepackage[colorlinks=true,allcolors=red,linktocpage=true]{hyperref}
\usepackage[utf8]{inputenc}
\usepackage{color}
\usepackage{comment}

\theoremstyle{definition}
\newtheorem{Definition}{Definition}[section]
\newtheorem{example}{Example}[section]
\newtheorem{remark}{Remark}[section]
\newtheorem{Assumption}{Assumption}[section]

\theoremstyle{plain}
\newtheorem{theorem}{Theorem}[section]
\newtheorem{lemma}{Lemma}[section]
\newtheorem{Corollary}{Corollary}[section]
\newtheorem{proposition}{Proposition}[section]

\begin{document}
	
\title{Flat affine symplectic Lie groups}

\author{Fabricio Valencia}

\date{\today}

\begin{abstract}
We give a new characterization of flat affine manifolds in terms of an action of the Lie algebra of classical infinitesimal affine transformations on the bundle of linear frames. We characterize flat affine symplectic Lie groups using symplectic \'etale affine representations and as a consequence of this, we show that a flat affine symplectic Lie group with bi-invariant symplectic connection contains a nontrivial one parameter subgroup formed by central translations. We give two methods for constructing flat affine symplectic Lie groups, thus obtaining all those having bi-invariant symplectic connections. We get nontrivial examples of simply connected flat affine symplectic Lie groups in every even dimension.
\end{abstract} 
\footnotetext{Partially Supported by CODI, Universidad de Antioquia.  Project Number 2015-7654.}
\maketitle


\vskip5pt
\noindent
Instituto de Matem\'aticas, Universidad de Antioquia, Medell\'in-Colombia

e-mail: fabricioyarro@gmail.com

\section{Introduction}
A connected and paracompact $2n$-dimensional symplectic manifold $(M,\omega)$ is called a \emph{flat affine symplectic manifold} if the bundle of symplectic frames $\textsf{Sp}(n,\mathbb{R})\hookrightarrow L(M)^\omega\stackrel{\pi}{\hbox to 18pt{\rightarrowfill}} M$ induced by $\omega$ admits a linear connection $\Gamma_\omega$ whose curvature and torsion 2-forms are identically null. To have a linear connection $\Gamma_\omega$ on $L(M)^\omega$ with the letter characteristics is equivalent to having a flat affine symplectic covariant derivative $\nabla$ on $(M,\omega)$, i.e.  both the curvature and torsion tensors of $\nabla$ vanish and $\nabla \omega=0$.  Moreover, the existence of a symplectic form $\omega$ and a flat affine symplectic covariant derivative $\nabla$ on $M$ is equivalent to have a maximal atlas for $M$ whose change of coordinates are restrictions of affine symplectic transformation of $\mathbb{R}^{2n}$ (see \cite{Fr}).

On the one hand, the study of symplectic manifolds equipped with symplectic connections can be motivated by their role in deformation quantization and geometric quantization (see \cite{FLD,Fe} and \cite{Hs,L}, respectively). A famous result of Fedosov gives a canonical deformation quantization which is defined by the data $(M,\omega,\nabla)$ where $\omega$ is a symplectic form and $\nabla$ is a torsion free symplectic connection on $M$. Such result motivates the definition and study of \emph{Fedosov manifolds} (see \cite{GRS}). On the other hand, some knowledge of the category of flat affine manifolds are necessary for understanding the category of Lagrangian submanifolds (see \cite[Thm 7.8]{W} and \cite{V,V2}). In fact, flat affine manifolds with holonomy reduced to $\textsf{GL}(m,\mathbb{Z})$ appear naturally in integrable
systems and Mirror Symmetry (see \cite{KS}). 
Further applications of  flat affine manifolds appear in the study of Hessian structures and information geometry (see \cite[c.\thinspace 6]{Sh}).

The main purpose of this paper is to study symplectic manifolds endowed with flat affine symplectic structures. More precisely, we want to study the structure of connected symplectic Lie groups which admit left invariant flat affine symplectic connections; some results on this topic can be found in \cite{An}, \cite{MR2}, \cite{Ba}, and \cite{NB}. The study of such a structures is motivated by the open problem proposed by J. Milnor in \cite{Mi} which asks to determine flat affine Lie groups (i.e., Lie groups which carry a left invariant
flat affine structure) and study their properties. The Lie algebra of a flat affine Lie group can be equipped with a structure of {\it left symmetric algebra}. This is a powerful object which we will be using throughout the paper. A very interesting survey paper where is discussed the origin, theory, and applications of left symmetric algebras (and hence flat affine Lie groups) in geometry and physics is \cite{Bu}.

The paper is divided as follows. In Section 2 we give a new characterization of flat affine manifolds (see Theorem \ref{F1}) which can be naturally used to characterize flat affine symplectic connections on the bundle of symplectic frames (see Proposition \ref{CSC}). This characterization is given in terms of an action of the Lie algebra of classical infinitesimal affine transformations on the bundle of linear frames of such a manifold. We show that such an infinitesimal action can be integrated (in the Lie-Palais' Theorem sense) in the case of compact parallelizable flat Riemannian manifolds (see Proposition \ref{F2}).

In Section 3 we introduce \emph{flat affine symplectic Lie groups}, which will be connected symplectic Lie groups endowed with left invariant flat affine symplectic connections. We introduce two results which appeared in a similar way for the pseudo-Riemannian case (see \cite{AuM}) but they have not been stated for the case of left invariant flat affine symplectic connections. These will be very useful for exploring the consequences of the constructions given in next sections. With the first one, we give a characterization of flat affine symplectic Lie groups in terms of ``symplectic'' \'etale affine representations (see Theorem \ref{CharacterizationLeft}). As consequence, the second one result allows us to conclude that every simply connected flat affine symplectic Lie group equipped with a bi-invariant symplectic connection can be identified with a Lie subgroup of affine symplectic transformations of its Lie algebra containing a nontrivial one parameter subgroup formed by central translations (see Proposition \ref{F16}). It is well known that a left invariant flat affine connection leaving parallel a left invariant volume form is geodesically complete if and only if the Lie group is unimodular (see \cite{Ba}). At the end of this section we give a different proof of this result for the case of left invariant flat affine symplectic connections. Our results put into evidence some properties of \emph{flat affine symplectic Lie algebras} (see Definition \ref{FASLA}). Further, other properties appear in \cite{MR2}, \cite{An}, and \cite{NB}.

In section 4, we give a method for constructing simply connected flat affine symplectic Lie groups using Nijenhuis' cohomology for left symmetric algebras (see \cite{N}). The construction will be called \emph{a double extension of a flat affine symplectic Lie algebra}. This is an iterative method which allows us to get flat affine symplectic Lie algebras of dimension $2n+2$ from a flat affine symplectic Lie algebra of dimension $2n$ (see Proposition \ref{F14} and Theorem \ref{F15}). We show that every simply connected flat affine symplectic Lie group whose Lie algebra is obtained as double extension of a flat affine symplectic Lie algebra with $\lambda=\mu$ (see Definition \ref{doubleextensionFAS}) can be identified with a subgroup of symplectic affine transformations of its Lie algebra containing a nontrivial one parameter subgroup formed by central translations (see Corollary \ref{F17}). We also prove that the Lie algebra of a symplectic Lie group equipped with a bi-invariant flat affine symplectic connection is obtained as a double extension of flat affine symplectic Lie algebras starting from $\lbrace 0 \rbrace$ (see Proposition \ref{F18}). Moreover, we observe that the Lie algebra of any flat affine symplectic Lie group of dimension $2$ can be obtained as a double extension starting from $\lbrace 0 \rbrace$. Flat affine symplectic Lie groups in dimension $2$ were found by A. Andrada in \cite{An} (see also \cite{MSG}). We also explain how to use this method to construct nontrivial simply connected flat affine symplectic Lie groups in every even dimension.

In section 5, we get flat affine symplectic Lie groups of dimension $2n$ from a connected flat affine Lie group of dimension $n$. We consider the classical cotangent symplectic Lie group of a connected flat affine Lie group defined in \cite{MR}, and show that there always exists a left invariant flat affine symplectic connection which parallelizes its two natural left invariant transverse Lagrangian foliations (see Proposition \ref{F19}). The corresponding Hess connection (compare \cite{Hs}) will be given explicitly. Again, using Nijenhuis' cohomology for left symmetric algebras, we construct the \emph{twisted cotangent symplectic Lie group} of a connected flat affine Lie group together with left invariant flat affine symplectic connections on it (see Proposition \ref{F20}). The reciprocal of the result obtained in Proposition \ref{F20} was proved by X. Ni and C. Bai in \cite{NB}. We give a shorter and different approach to get it which allows us to show more consequences of our construction. It is important to notice that this construction will be parametrized by a commutative product over the corresponding left symmetric algebra and a 2-cocycle of this left symmetric algebra with values in its dual vector space. We observe that when both the commutative product and the 2-cocycle are null we obtain the classical cotangent symplectic Lie group defined in \cite{MR}. Moreover, in this case the flat affine symplectic connection is the Hess connection (see Corollary \ref{F21}). We prove that on a simply connected Lie group there exists the structure of flat affine symplectic Lie group which admits a normal Abelian Lagrangian Lie subgroup if and only if, the group is isomorphic to the twisted cotangent symplectic Lie group of a connected flat affine Lie group (see Theorem \ref{F22}). Finally, we give necessary and sufficient conditions to say when the left invariant flat affine symplectic connections on the twisted cotangent symplectic Lie group are geodesically complete (see Proposition \ref{F23}).

\section{New characterization of flat affine manifolds}
In this short section we give a new characterization of flat affine manifolds in terms of an action of the Lie algebra of classical infinitesimal affine transformations on the bundle of linear frames. As an immediate consequence of such result we characterize flat Riemannian metrics, flat affine connections leaving parallel a volume form, and flat affine symplectic connections. We show that such an infinitesimal action can be integrated (in the Lie-Palais' Theorem sense) in the case of compact parallelizable flat Riemannian manifolds.

In what follows $M$ denotes a connected and paracompact $n$-dimensional smooth manifold without boundary, $P=L(M)$ its bundle of linear frames, $\theta$ the canonical 1-form, $\Gamma$ a linear connection on $P$ of connection 1-form $\mathcal{A}$, and $\nabla$ the covariant derivative on $M$ associated to $\Gamma$. Consider the right action of $\textsf{GL}(n,\mathbb{R})$ on $P$.  Denote by $H^\ast$ the fundamental vector field associated to an element $H\in\mathfrak{gl}(n,\mathbb{R})$ and $B(\xi)$ the standard horizontal vector field associated to $\xi\in\mathbb{R}^n$. This vector field is determined by the relation $\theta(B(\xi))=\xi$. Recall that the connection 1-form $\mathcal{A}$ associated to $\Gamma$ is a $\mathfrak{gl}(n,\mathbb{R})$-valued $1$-form over $P$ satisfying
$$\mathcal{A}(H^\ast)=H,\qquad H\in\mathfrak{gl}(n,\mathbb{R})\quad \text{and}$$
$$R_a^*\mathcal{A}=\text{Ad}_{a^{-1}}\circ \mathcal{A},\qquad a\in\textsf{GL}(n,\mathbb{R}),$$
where $R_a$ is the action of $a\in\textsf{GL}(n,\mathbb{R})$ on $P$. Furthermore, the canonical  1-form $\theta$ is an $\mathbb{R}^n$-valued tensorial 1-form over $P$ of type $(\textsf{GL}(n,\mathbb{R}),\mathbb{R}^n)$.
Also, we have
$$(R_a)_\ast B(\xi)=B(a^{-1}\xi),\qquad a\in\textsf{GL}(n,\mathbb{R}),\quad \xi\in\mathbb{R}^n.$$
It is well known that the curvature 2-form $\Omega_\mathcal{A}\in\Omega^2(P,\mathfrak{gl}(n,\mathbb{R}))$ and the torsion 2-form $\Theta\in\Omega^2(P,\mathbb{R}^n)$ associated to the connection 1-form $\mathcal{A}$  respectively satisfy the Cartan's structure equations
$$\text{d}\mathcal{A}(X,Y)=-\dfrac{1}{2}[\mathcal{A}(X),\mathcal{A}(Y)]+\Omega_\mathcal{A}(X,Y)\quad\text{and}$$
$$
\text{d}\theta(X,Y)=-\dfrac{1}{2}(\mathcal{A}(X)\cdot \theta(Y)-\mathcal{A}(Y)\cdot\theta(X))+
\Theta(X,Y),$$
for all $X,Y\in T_u P$ with $u\in P$. When $\Omega_\mathcal{A}=0$ and $\Theta=0$ we say that $\Gamma$ is a \emph{flat affine connection}.

It is simple to check that a connection $\Gamma$ on $P$ is flat affine if and only if both the curvature and torsion tensors of the corresponding covariant derivative $\nabla$ are null. The pair $(M,\nabla)$ is called a \emph{flat affine manifold} if $\nabla$ is a flat affine connection on $M$. That $(M,\nabla)$ is a flat affine manifold is equivalent to the existence of a maximal atlas of $M$ whose change of coordinates are restrictions of affine transformations of $\mathbb{R}^n$ (see \cite{AM}). The set of affine transformation of $\mathbb{R}^n$, denoted by $\text{Aff}(\mathbb{R}^n)$, is a Lie group isomorphic to the semi-direct  product $\mathbb{R}^n\rtimes_{Id}\textsf{GL}(n,\mathbb{R})$. The Lie algebra of $\textsf{Aff}(\mathbb{R}^n)$ is the product vector space $\mathfrak{aff}(\mathbb{R}^n)= \mathbb{R}^n\rtimes_{id} \mathfrak{gl}(n,\mathbb{R})$ with Lie bracket
$$[(\xi,H),(\xi',H')]=(H\xi'-H'\xi,[H,H'])\quad H,H'\in\mathfrak{gl}(n,\mathbb{R}),\quad\xi,\xi'\in\mathbb{R}^n.$$

In these terms, we have the following characterization of flat affine manifolds.
\begin{theorem}\label{F1}
	A connection $\Gamma$ on $P$ is flat affine, if and only if
	\begin{align*}
	\widetilde{\eta}: \mathfrak{aff}(\mathbb{R}^n) &\longrightarrow \mathfrak{X}(P)\\
	(\xi,H) &\longmapsto B(\xi)+H^\ast,
	\end{align*}
	is an infinitesimal action of $\mathfrak{aff}(\mathbb{R}^n)$ over $P$.
\end{theorem}
\begin{proof}
	It is well known that the map $\mathfrak{gl}(n,\mathbb{R})\to \mathfrak{X}(P)$, defined by $H\mapsto H^\ast$, is a Lie algebra homomorphism. Hence
	$$
	\widetilde{\eta}([(\xi,H),(\xi',H')])=B(H\xi')-B(H'\xi)+[H^\ast,H'^\ast].
	$$
	for all $(\xi,H)$ and $(\xi',H')$ in $\mathfrak{aff}(\mathbb{R}^n)$. Moreover, for all $H\in \mathfrak{gl}(n,\mathbb{R})$ and $\xi\in\mathbb{R}^n$, we have $[H^\ast,B(\xi)]=B(H\xi)$. Thus
	$$[\widetilde{\eta}(\xi,H),\widetilde{\eta}(\xi',H')]=[B(\xi),B(\xi')]+B(H\xi')-B(H'\xi)+[H^\ast,H'^\ast].$$
	Therefore, the map $\widetilde{\eta}$ defines an infinitesimal action of $\mathfrak{aff}(\mathbb{R}^n)$ over $P$ if and only if $[B(\xi),B(\xi')]=0$ for all $\xi,\xi'\in\mathbb{R}^n$.
	
	If $\Gamma$ is a flat affine connection, then $[B(\xi),B(\xi')]$ is a vertical (respectively $[B(\xi),B(\xi')]$ is a horizontal) vector field (see \cite[p. 136]{KN}). Thus, $[B(\xi),B(\xi')]=0$ for all $\xi,\xi'\in\mathbb{R}^n$. Conversely, it is well known that if $X$ is a vertical vector at $u\in P$ there exists an element $H\in\mathfrak{gl}(n,\mathbb{R})$ such that $X=H^\ast_u$. On the other hand, if $Y$ is a horizontal vector at $u\in P$ there exists $\xi\in\mathbb{R}^n$ so that $Y=B(\xi)_u$. Therefore, making use of the Cartan's structure equations for $\Gamma$, a simple computation allows us to show that considering the cases when: $X$ and $Y$ are vertical or horizontal or one is vertical and the another is horizontal, then the condition $[B(\xi),B(\xi')]=0$ for all $\xi,\xi'\in\mathbb{R}^n$ implies directly that $\Gamma$ is a flat affine connection on $P$.
\end{proof}
\begin{remark}
	Under the assumptions of Theorem \ref{F1} it is easy to verify that the infinitesimal action $\widetilde{\eta}$ is effective.
\end{remark}
\begin{Corollary}\label{CorollaryCharacterization}
	\begin{enumerate}
		\item Let $(M,g)$ be a Riemannian manifold and $\textsf{O}(n,\mathbb{R})\hookrightarrow L(M)^g\stackrel{\pi}{\hbox to 18pt{\rightarrowfill}} M$ the bundle of orthonormal frames over $M$ induced by $g$. The Levi-Civita connection $\Gamma^g$ on $L(M)^g$ is flat if and only if the restriction of $\widetilde{\eta}$ to $\mathbb{R}^n\rtimes_{id} \mathfrak{o}(n,\mathbb{R})$ is an infinitesimal action of $\mathbb{R}^n\rtimes_{id} \mathfrak{o}(n,\mathbb{R})$ over $L(M)^g$.
		\item  If $\sigma$ is a volume form on $M$ and $\textsf{SL}(n,\mathbb{R})\hookrightarrow L(M)^\sigma\stackrel{\pi}{\hbox to 18pt{\rightarrowfill}} M$ is the bundle of special frames over $M$ determined by $\sigma$, then a linear connection $\Gamma^\sigma$ on $L(M)^\sigma$ is flat affine if and only if the restriction of $\widetilde{\eta}$ to $\mathbb{R}^n\rtimes_{id} \mathfrak{sl}(n,\mathbb{R})$ is an infinitesimal action of $\mathbb{R}^n\rtimes_{id} \mathfrak{sl}(n,\mathbb{R})$ over $L(M)^\sigma$.
	\end{enumerate}
\end{Corollary}
It is important to have in mind that linear connections on $L(M)^\sigma$ are in bijective correspondence with covariant derivatives $\nabla$ on $(M,\sigma)$ such that $\nabla \sigma=0$.
\begin{proposition}\label{F2}
	Let $(M,g)$ be a compact and parallelizable flat Riemannian manifold. Then, the infinitesimal action $\widetilde{\eta}_g: \mathbb{R}^n\rtimes_{id} \mathfrak{o}(n,\mathbb{R}) \to \mathfrak{X}(L(M)^g)$ of Corollary \ref{CorollaryCharacterization} is integrable.
\end{proposition}
\begin{proof}
	If $\lbrace X_1,\cdots, X_n \rbrace$ is a parallelism of $M$, then by the Gram--Schmidt orthogonalization process we can construct an orthonormal parallelism $\lbrace \widetilde{X_1},\cdots,\widetilde{X_n}\rbrace$ of $M$. Thus, $\textsf{O}(n,\mathbb{R})\hookrightarrow L(M)^g\stackrel{\pi}{\hbox to 18pt{\rightarrowfill}} M$ is isomorphic to the trivial $\textsf{O}(n,\mathbb{R})$-principal bundle $\textsf{O}(n,\mathbb{R})\times M$. Hence, as $\textsf{O}(n,\mathbb{R})\times M$ is a compact smooth manifold, if 
	$G$ is a simply connected Lie group with Lie algebra $\mathbb{R}^n\rtimes_{id} \mathfrak{o}(n,\mathbb{R})$, then by the Lie-Palais' Theorem there exists a unique smooth right action of $G$ on $L(M)^g$ such that the fundamental vector field associated to $(\xi,H)\in\mathbb{R}^n\rtimes_{id} \mathfrak{o}(n,\mathbb{R})$ is $\widetilde{\eta}_g(\xi,H)$.
\end{proof}
As every Lie group is parallelizable, then circle $S^1$ and $n$-torus $\mathbb{T}^n$ are examples of compact parallelizable flat Riemannian manifolds, hence both satisfy the Proposition \ref{F2}.

\section{Flat affine symplectic Lie groups}
To determine flat affine Lie groups (i.e., Lie groups which carry a left invariant flat affine structure) is an open problem proposed by J. Milnor in \cite{Mi}. From now on, our purpose is to study this problem in the case of connected symplectic Lie groups endowed with left invariant flat affine symplectic connections. In this section, we treat some properties of these kind of Lie groups which will be useful throughout the paper. 
\subsection{Flat affine symplectic manifolds}
Let $(M,\omega)$ be a symplectic manifold of dimension $2n$. A symplectic frame at $p\in M$ is a symplectic ordered basis of $(T_pM,\omega_p)$. Let us denote by $L(M)^\omega$ the set of all symplectic frames at all points of $M$ and by $\pi$ the natural projection of $L(M)^\omega$ onto $M$. Using Darboux's Theorem, it is possible to determinate a differentiable structure over $L(M)^\omega$ so that the map $\pi$ is smooth and $L(M)^\omega$ has a natural structure of $\textsf{Sp}(n,\mathbb{R})$-principal bundle. Consider the inclusion map $\iota:L(M)^\omega \hookrightarrow P$, the inclusion group homomorphism $\varphi: \textsf{Sp}(n,\mathbb{R})\hookrightarrow\textsf{GL}(2n,\mathbb{R})$, and the identity map $\text{Id}_M$ of $M$. It is simple to check that $(\iota,\varphi,\text{Id}_M)$ determines a homomorphism of principal bundles between $L(M)^\omega$ and $P$. Thus, every connection $\Gamma_\omega$ on $L(M)^\omega$ determines a unique linear connection $\Gamma$ on $P$ (see \cite[p. 79]{KN}). A connection $\Gamma_\omega$ on $P^\omega$ is called a \emph{linear symplectic connection}. If we denote by $(e_1,\cdots,e_{2n})$ to the canonical symplectic basis of $(\mathbb{R}^{2n},\omega_0)$, and take $u=(X_1,\cdots,X_n,\widetilde{X_1},\cdots,\widetilde{X_n})$ as a symplectic frame at $p\in M$, then the map $u:(\mathbb{R}^{2n},\omega_0)\to (T_pM,\omega_p)$ defined by $u(e_j)=X_j$ and $u(e_{j+n})=\widetilde{X_j}$ for all $j=1,2,\cdots,n$, is an isomorphism of symplectic vector spaces. We define the $\mathbb{R}^{2n}$-valued 1-form $\theta_\omega$ on $P^\omega$ by
$$\theta_\omega(X):=u^{-1}(\pi_{\ast,u}(X)),\qquad X\in T_u P^\omega.$$

If $\Gamma_\omega$ is a connection on $L(M)^\omega$, then the torsion 2-form $\Theta_\omega$ associated to the canonical 1-form $\theta_\omega$ is defined as the exterior covariant differential of $\theta_\omega$ with respect to $\Gamma_\omega$ and it is also completely determined by the second Cartan's structure equation. When both the curvature 2-form $\Omega_{\mathcal{A}_\omega}$\footnote{Here $\mathcal{A}_\omega$ denotes the connection 1-form associated to $\Gamma_\omega$.} and the torsion 2-form $\Theta_\omega$ are null, we say that $\Gamma_\omega$ is a \emph{flat affine symplectic connection}. As a consequence of Theorem \ref{F1} we have
\begin{proposition}\label{CSC}
	Let $(M,\omega)$ be a $2n$-dimensional symplectic manifold and $\textsf{Sp}(n,\mathbb{R})\hookrightarrow L(M)^\omega\stackrel{\pi}{\hbox to 18pt{\rightarrowfill}} M$ the bundle of symplectic frames over $M$ induced by $\omega$. A symplectic linear connection $\Gamma^\omega$ on $L(M)^\omega$ is flat affine if and only if the restriction of $\widetilde{\eta}$ to $\mathbb{R}^n\rtimes_{id} \mathfrak{sp}(n,\mathbb{R})$ is an infinitesimal action of $\mathbb{R}^n\rtimes_{id} \mathfrak{sp}(n,\mathbb{R})$ over $L(M)^\omega$.
\end{proposition}
A covariant derivative $\nabla$ on a symplectic manifold $(M,\omega)$ is called \emph{symplectic} if $\nabla \omega =0$, that is
$$X\cdot \omega(Y,Z)=\omega(\nabla_XY,Z)+\omega(Y,\nabla_XZ),\qquad X,Y,Z\in \mathfrak{X}(M).$$
\begin{remark}
	\begin{enumerate}
		\item There exists a bijective correspondence between linear symplectic connections $\Gamma_\omega$ on $L(M)^\omega$ and symplectic covariant derivatives on $(M,\omega)$.
		\item Given a paracompact symplectic manifold $(M,\omega)$, it is always possible to ensure that there exists a torsion free symplectic connection on $M$ but unlike the Riemannian case this is not unique (see \cite{V}).
		\item The holonomy group of a symplectic connection on a connected symplectic manifold is identified with a subgroup of $\textsf{Sp}(n,\mathbb{R})$ defined up to conjugation. Moreover, there is always a natural surjective group homomorphism between the fundamental group $\Pi_1(M)$ and $\text{Sp}(n,\mathbb{R})/G$ where $G$ is a normal Lie subgroup of $\textsf{Sp}(n,\mathbb{R})$ defined up to conjugation.
		\item If $(M,\omega)$ is a compact symplectic manifold endowed with a flat affine symplectic connection $\nabla$, then the Euler characteristic of $M$ vanishes. This follows from the fact that the volume form induced on $M$ by $\omega$ is parallel with respect to $\nabla$. This result is a particular case of Chern's conjecture for compact flat affine manifolds and it is a direct consequence of the case proved by B. Klingler when the compact flat affine manifold admits a parallel volume form (see \cite{K}).
	\end{enumerate}
\end{remark}
\begin{Definition}
	A triple $(M,\omega,\nabla)$ where $(M,\omega)$ is a symplectic manifold and $\nabla$ is a flat affine symplectic connection on $M$ is called a \emph{flat affine symplectic manifold}.
\end{Definition}
To have a flat affine symplectic structure on a $2n$-dimensional manifold $M$ is equivalent to having a maximal atlas of $M$ whose change of coordinates are restriction of elements of $\mathbb{R}^{2n}\rtimes_{Id}\textsf{Sp}(n,\mathbb{R})$ (see \cite{Fr}). Two interesting examples of flat affine symplectic manifolds are the 2-torus $\mathbb{T}^2$ (see \cite{Ku} or \cite[p. 211]{KN}) and the ordinary cylinder $S^1\times \mathbb{R}$ (compare \cite[p. 223]{KN}).
\subsection{Flat affine symplectic Lie groups}
In what follows $G$ denotes a connected real Lie group and $\mathfrak{g}:=T_\epsilon G$ its Lie algebra. A symplectic form $\omega^+$ (respectively a linear connection $\nabla$) on $G$ is called \emph{left invariant} if $L_\sigma:G\to G$, defined by $\tau\mapsto \sigma\tau$, is a symplectomorphism of $(G,\omega^+)$ (respectively an affine transformation of $(G,\nabla)$) for all $\sigma\in G$. 
\begin{Definition}
	\begin{enumerate}
		\item \cite{LM}, \cite{MR} A pair $(G,\omega^+)$ where $\omega^+$ is a left invariant symplectic form on $G$ is called a \emph{symplectic Lie group}.
		\item  \cite{Ko}, \cite{Vi}, \cite{M} A pair $(G,\nabla)$ where $\nabla$ is a left invariant flat affine connection on $G$ is called a \emph{flat affine Lie group}.
	\end{enumerate}
\end{Definition}

From now on we will deal with the following object:
\begin{Definition}
	A triple $(G,\omega^+,\nabla)$ where $(G,\omega^+)$ is a symplectic Lie group and $\nabla$ is a left invariant flat affine connection on $G$ which is symplectic with respect to $\omega^+$ is called a \emph{flat affine symplectic Lie group}.
\end{Definition}
If $(G,\omega^+)$ is a symplectic Lie group and $\nabla$ is a left invariant symplectic connection on $G$, then
$$ \omega^+(\nabla_{x^+}y^+,z^+)+\omega^+(y^+,\nabla_{x^+}z^+)=0,\qquad x,y,z\in\mathfrak{g}.
$$
Here $x^+$ denotes the left invariant vector field associated to $x\in\mathfrak{g}$.
\begin{remark}\label{inducedconnection}
	Every symplectic Lie group $(G,\omega^+)$ is also flat affine. More precisely, the left invariant symplectic form $\omega^+$ defines on $G$ a unique natural left invariant flat affine connection $\nabla$ determined by the formula
	$$\omega^+(\nabla_{x^+}y^+,z^+)=-\omega^+(y^+,[x^+,z^+]),\qquad x,y,z\in\mathfrak{g}.$$
	This result was proved by A. Bon-Yau Chu in \cite{B}. It is easy to show that $\nabla$ is symplectic with respect to $\omega^+$ if and only if $G$ is commutative.
\end{remark}

Some results and classifications of flat affine symplectic Lie groups are already known in the literature; see for instance  \cite{An}, \cite{MR2}, \cite{Ba}, and \cite{NB}. In this paper we want to exhibit other interesting properties of these kind of Lie groups as well as give two methods for constructing simply connected flat affine symplectic Lie groups in every even dimension. We will do it using the algebraic structure and the properties of flat affine symplectic Lie algebras. These objects are defined below.

The following two results will be vital for exploring the consequences of the constructions given in next sections. First of all, we give the following characterization of flat affine symplectic Lie groups using \'etale affine representations\footnote{Let $V$ be an $n$-dimensional vector space. A Lie group homomorphism $\rho:G\to\textsf{Aff}(V)$ is called an \emph{\'etale affine representation} if the left action of $G$ on $V$ determined by $\rho$ admits a point with open orbit and discrete isotropy.}. Affirmation (1) if and only if (2) is well known (see for instance \cite{MR2} and \cite{NB}).
\begin{theorem}\label{CharacterizationLeft}
	Let $G$ be a connected Lie group of dimension $2n$, $\mathfrak{g}$ its Lie algebra, and $\widetilde{G}$ its universal covering Lie group. The following are equivalent
	\begin{enumerate}
		\item There exist both $\omega^+$ and $\nabla$ on $G$ such that $(G,\omega^+,\nabla)$ is a flat affine symplectic Lie group.
		\item There exist both a nondegenerate scalar 2-cocycle\footnote{$\omega\in \bigwedge^2\mathfrak{g}^*$ is a scalar 2-cocycle on $\mathfrak{g}$ if $\omega([x,y],z)+\omega([y,z],x)+\omega([z,x],y)=0$ for all $x,y,z\in\mathfrak{g}.$} $\omega\in \bigwedge^2\mathfrak{g}^*$ and a bilinear map $\cdot:\mathfrak{g}\times\mathfrak{g}\to \mathfrak{g}$ over $\mathfrak{g}$ satisfying
		\begin{eqnarray}
		&  & \omega(L_x(y),z)+\omega(y,L_x(z))=0, \nonumber\\
		&  & L_{[x,y]}=[L_x,L_y]_{\mathfrak{gl}(\mathfrak{g})},\quad\text{and}\label{left1}\\
		&  & [x,y]=L_x(y)-L_y(x),\qquad x,y,z\in\mathfrak{g},\label{left2}
		\end{eqnarray}
		where $L_x:\mathfrak{g}\to\mathfrak{g}$ is the linear map defined by $L_x(y):=x\cdot y$.
		\item There exists a Lie group homomorphism $\rho: \widetilde{G}\to V\rtimes_{Id}\textsf{Sp}(V,\omega)$ where $(V,\omega)$ is a real $2n$-dimensional symplectic vector space, such that the natural left action of $\widetilde{G}$ over $V$ admits a point with open orbit and discrete isotropy.
	\end{enumerate}
\end{theorem}
\begin{proof}
	If $(G,\omega^+,\nabla)$ is a flat affine symplectic Lie group, it is easy to check that $\omega:=\omega^+_\epsilon$ is a nondegenerate scalar 2-cocycle of $\mathfrak{g}$ and $\cdot:\mathfrak{g}\times \mathfrak{g}\to\mathfrak{g}$ defined by $x\cdot y=L_x(y):=(\nabla_{x^+}y^+)(\epsilon)$ is a bilinear product on $\mathfrak{g}$ which satisfies the three identities of item (2).
	
	Now suppose (2), thus the map $\theta\colon \mathfrak{g}\to \mathfrak{g}\rtimes_{id}\mathfrak{sp}(\mathfrak{g},\omega)$ 
	defined by $x\mapsto(x,L_x)$
	is a well defined Lie algebra homomorphism.  Passing to exponential, we get a Lie group homomorphism $\rho\colon \widetilde{G} \to \mathfrak{g}\rtimes_{Id}\textsf{Sp}(\mathfrak{g},\omega)$ sending $\sigma=\textsf{exp}_G(x)$ to $\rho(\sigma)=(Q(\sigma),F_\sigma)$ where
	$$Q(\sigma)=\sum_{k=1}^\infty \dfrac{1}{k!}(L_x)^{k-1}(x)\quad\text{and}\quad F_\sigma=\textsf{Exp}(L_x)=\sum_{k=0}^\infty \dfrac{1}{k!}(L_x)^k.$$
	Thus, as the map $\psi_0:\mathfrak{g}\to\mathfrak{g}$ defined by $x\mapsto \theta(x)(0)=x+L_x(0)=x$ is a linear isomorphism, we have that $0\in\mathfrak{g}$ is a point with open orbit and discrete isotropy for the left action of $\widetilde{G}$ over $\mathfrak{g}$ determined by $\rho$.
	
	Finally, suppose that $\rho\colon \widetilde{G}\to V\rtimes_{Id}\textsf{Sp}(V,\omega)$ 
	is a homomorphism of  Lie groups such that the orbital map  
	$\pi\colon\widetilde{G}\to \text{Orb}(v)$ defined by 
	$\sigma\mapsto Q(\sigma)+F_\sigma(v)$ is a local diffeomorphism for some  
	$v\in V$. Differentiating at the identity of $\widetilde{G}$, 
	we obtain a Lie algebra homomorphism 
	$\theta\colon \mathfrak{g} \to \mathfrak{g}\rtimes_{id}\mathfrak{sp}(V,\omega)$ given by  
	$x \mapsto (q(x),f_x)$, where the linear map  $\psi_v\colon \mathfrak{g} \to V$ defined by 
	$x \mapsto q(x)+f_x(v)$ is a linear isomorphism (see \cite{M}). Moreover, the map $f:\mathfrak{g}\to \mathfrak{sp}(V,\omega)$ defined by $x\to f_x$ is also a Lie algebra homomorphism and $q:\mathfrak{g}\to V$ given by $x\to q(x)$ is a linear map such that 
	\begin{equation}\label{1cocycle}
	q([x,y])=f_x(q(y))-f_y(q(x)), \qquad x,y\in\mathfrak{g}.
	\end{equation}
	Now, we define on $\mathfrak{g}$ the skew-symmetric bilinear form
	$\widetilde{\omega}$ and the bilinear map $\cdot:\mathfrak{g}\times \mathfrak{g}\to \mathfrak{g}$ respectively by
	$$\widetilde{\omega}(x,y)=\omega(\psi_v(x),\psi_v(y))$$
	and
	$$L_x=\psi_v^{-1}\circ f_x\circ \psi_v,\quad y\mapsto x\cdot y=L_x(y),$$
	for all $x,y\in\mathfrak{g}$. Because $f:\mathfrak{g}\to \mathfrak{sp}(V,\omega)$ is a Lie algebra homomorphism, 
	we have that $L_{[x,y]}=[L_x,L_y]_{\mathfrak{gl}(\mathfrak{g})}$. On the other hand, since 
	$q:\mathfrak{g}\to V$ satisfies \eqref{1cocycle}, we conclude that
	$[x,y]=L_x(y)-L_y(x)$. Moreover, the fact that $f_x\in \mathfrak{sp}(V,\omega)$ implies
	\begin{equation}\label{2cocycletilde}
	\widetilde{\omega}(L_x(y),z)+\widetilde{\omega}(y,L_x(z))=0,\qquad x,y,z\in\mathfrak{g}.
	\end{equation}
	
	It follows that the identity \eqref{2cocycletilde} implies that $\widetilde{\omega}$ is a scalar 2-cocycle of $\mathfrak{g}$. Hence, it is simple to check that the left invariant symplectic form 
	$$\omega^+_\sigma(X_\sigma,Y_\sigma):=\widetilde{\omega}((L_{\sigma^{-1}})_{\ast,\sigma}X_\sigma,(L_{\sigma^{-1}})_{\ast,\sigma}Y_\sigma),\quad\sigma\in G,\quad X_\sigma,Y_\sigma\in T_\sigma G,$$
	and the left invariant flat affine connection obtained after extending
	$$\nabla_{x^+}y^+:=(x\cdot y)^+=(L_x(y))^+,\qquad x,y\in\mathfrak{g},$$
	are such that $(G,\omega^+,\nabla)$ is a flat affine symplectic Lie group.
\end{proof}
It is easy to verify that the identities \eqref{left1} and \eqref{left2} of item (2) in the previous theorem imply that
\begin{equation}\label{SGprodut}
x\cdot(y\cdot z)-(x\cdot y)\cdot z = y\cdot(x\cdot z)-(y\cdot x)\cdot z,\qquad x,y,z\in\mathfrak{g}.
\end{equation}

To set up terminology we give the following definition.
\begin{Definition}\cite{Vi}
	A bilinear map over a vector space satisfying the formula \eqref{SGprodut} is called a \emph{left symmetric product} and the corresponding vector space is called a \emph{left symmetric algebra}.
\end{Definition}

Some important geometric consequences of studying left symmetric algebras can be found in \cite{Ki}, \cite{B}, \cite{N}, \cite{M}, and \cite{Bu}. Next object will be relevant for doing the constructions of some simply connected flat affine symplectic Lie groups in the following sections: 
\begin{Definition}\label{FASLA}
	A triple $(\mathfrak{g},\omega,\cdot)$ where $\mathfrak{g}$ is a real finite dimensional Lie algebra, $\omega$ is a nondegenerate scalar 2-cocycle of $\mathfrak{g}$, and $\cdot:\mathfrak{g}\times\mathfrak{g}\to \mathfrak{g}$ is a left symmetric product over $\mathfrak{g}$ whose commutator agrees with the Lie bracket of $\mathfrak{g}$ and it verifies
	\begin{equation}\label{n12.11}
	\omega(x\cdot y,z)+\omega(y,x\cdot z)=0,\qquad x,y,z\in\mathfrak{g},
	\end{equation}
	is called a \emph{flat affine symplectic Lie algebra}.
\end{Definition}
Suppose that $(G,\omega^+,\nabla)$ is a flat affine symplectic Lie group. We can show a weaker claim than item (3) of Theorem \ref{CharacterizationLeft} as follows. Let $(\mathfrak{g},\omega,\cdot)$ be the flat affine symplectic Lie algebra associated to $(G,\omega^+,\nabla)$ and denote by $\mathfrak{g}^\ast$ the dual vector space of $\mathfrak{g}$. If $L^\ast:\mathfrak{g}\to \mathfrak{gl}(\mathfrak{g}^\ast)$ denotes the dual representation of $L$, then a direct computation allows us to show that 
$$\omega([x,y],\cdot)=L^*_x(\omega(y,\cdot))-L^*_y(\omega(x,\cdot)),\qquad x,y\in\mathfrak{g}.$$
Thus, the map $\theta: \mathfrak{g} \to \mathfrak{aff}(\mathfrak{g}^*)$ defined by $x \mapsto (\omega(x,\cdot),L^*_x)$, is a well defined Lie algebra homomorphism. Moreover, as $\omega$ is nondegenerate, the map $\psi_0:\mathfrak{g} \to \mathfrak{g}^*$ defined by $x\mapsto \psi_0(x):=\omega(x,\cdot)$ is a linear isomorphism. Therefore, there exists a Lie group homomorphism $\rho:\widetilde{G}\to\textsf{Aff}(\mathfrak{g}^\ast)$ such that $0\in\mathfrak{g}^\ast$ is a point with open orbit and discrete isotropy. Consequently, the orbital map $\pi:\widetilde{G}\to\text{Orb}(0)\subset \mathfrak{g}^\ast$ which is determined by
$$\pi(\textsf{exp}_G(x))=\sum_{k=1}^\infty \dfrac{1}{k!}(L^*_x)^{k-1}(\omega(x,\cdot)),$$
is a covering map.

\begin{remark}
	If $(G,\omega^+,\nabla)$ is a connected flat affine symplectic Lie group, then the \'etale affine representation $\rho:\widetilde{G}\to \mathfrak{g}\rtimes_{Id}\textsf{Sp}(\mathfrak{g},\omega^+_\epsilon
	)$ obtained by means of the exponential map of $G$ and the Lie algebra homomorphism $\theta:\mathfrak{g} \to \mathfrak{g}\rtimes_{id}\mathfrak{sp}(\mathfrak{g},\omega^+_\epsilon
	)$ which is defined by $x\mapsto(x,L_x)$, is injective. Thus, we can identify the Lie group $\widetilde{G}$ with the subgroup $\rho(\widetilde{G})$ of $\mathfrak{g}\rtimes_{Id}\textsf{Sp}(\mathfrak{g},\omega^+_\epsilon
	)$.
\end{remark}
Let $G$ be a simply connected Lie group. To have a left invariant flat affine connection $\nabla$ on $G$ is equivalent to having an \'etale affine representation $\rho:G\to\textsf{Aff}(V)$ (see \cite{Ko}). If $v\in V$ is a point with open orbit and discrete isotropy, then the completeness of $\nabla$ is equivalent to have that such an action being simply transitive or equivalently that $\text{Orb}(v)=V$ and the orbital map $\pi:G\to \text{Orb}(v)$ is a global diffeomorphism (see \cite{FGH} or \cite{M} for more details). L. Auslander conjectured in \cite{Aus} that a simply transitive action contains nonzero translations if $G$ is nilpotent, that is, the restriction to $\rho(G)\subset \textsf{Aff}(V)$ of the natural homomorphism $\lambda:\textsf{Aff}(V)\to \textsf{GL}(V)$, is not injective. J. Scheuneman gave a presumed proof of this conjecture in \cite{Sc} but D. Fried observed in \cite{F} that his proof had an error and exhibited a counterexample showing that the Auslander's conjecture is not true in general. Later A. Medina and Yu. Khakimdjanov in \cite{MK} exhibited more general examples on flat affine filiform Lie groups, for which Auslander's conjecture is true in odd dimension and false in even dimension.

With the aim of presenting a result related to this story, we consider bi-invariant flat affine symplectic connections. A linear connection $\nabla$ on $G$ in called \emph{bi-invariant} if both $L_\sigma: G\to G$ and $R_\sigma: G\to G$ (defined by $\tau\mapsto \tau\sigma$) are affine transformations of $(G,\nabla)$ for all $\sigma\in G$. It is well known that a Lie group admits a bi-invariant flat affine connection if and only if the Lie bracket of $\mathfrak{g}$ is underlying of an associative product on $\mathfrak{g}$ (see for instance \cite{M}). 
\begin{proposition}\label{F16}
	Let $(G,\omega^+)$ be a simply connected symplectic Lie group and assume that it admits a bi-invariant flat affine symplectic connection. Then $G$ is a nilpotent Lie group and it can be identified with a subgroup of $\mathfrak{g}\rtimes_{Id}\textsf{Sp}(\mathfrak{g},\omega^+_\epsilon
	)$ containing a nontrivial one parameter subgroup formed by central translations.
\end{proposition}
\begin{proof}
	Let $\nabla$ be a bi-invariant flat affine symplectic connection on $(G,\omega^+)$. Then, the bilinear map $\cdot:\mathfrak{g}\times\mathfrak{g}\to\mathfrak{g}$ defined by $x\cdot y=L_x(y)=(\nabla_{x^+}y^+)(\epsilon)$ for all $x,y\in\mathfrak{g}$, determines an associative product over $\mathfrak{g}$ satisfying \eqref{left2} and \eqref{n12.11}. As $\cdot$ is associative, we have that $L_x\circ L_y=L_{x\cdot y}$ for all $x,y\in\mathfrak{g}$. This implies that $L_x\circ L_y\in \mathfrak{sp}(\mathfrak{g},\omega^+_\epsilon)$. If $(L_x\circ L_y)^+$ denotes adjoint map of $L_x\circ L_y$ with respect to $\omega$, then we have that $-(L_x\circ L_y)=(L_x\circ L_y)^+=L_y\circ L_x$, that is, $L_x\circ L_y+L_y\circ L_x=0$ for all $x,y\in\mathfrak{g}$. Therefore, $L_x\circ L_x=0$ for every $x\in\mathfrak{g}$. Note that if $\mathfrak{g}$ is a non-nilpotent associative algebra, then there is a nonzero element $a\in\mathfrak{g}$ which is idempotent, that is, $a\cdot a=a$ with $a\neq 0$ (see for instance \cite[p. 23]{A}). Thus
	$$0=(L_a\circ L_a)(a)=L_a(a\cdot a)=L_a(a)=a\cdot a=a\neq 0,$$
	which is a contradiction. Therefore, $(\mathfrak{g},\cdot)$ is an associative nilpotent algebra. Consequently, there is an integer $m\geq 2$ such that
	$$\mathfrak{g}^m=\lbrace x_1\cdot x_2\cdots x_m:\quad x_j\in \mathfrak{g},\quad \forall j=1,2,\cdots,m\rbrace=\lbrace 0\rbrace$$
	with $\mathfrak{g}^{m-1}\neq\lbrace 0\rbrace$. This implies that there exists $b\neq 0$ such that $L_b=R_b=0$ where $R_b(x)=x\cdot b$. As $G$ is simply connected, by Theorem \ref{CharacterizationLeft} we have that there is a Lie group homomorphism $\rho$ between $G$ and the Lie group $\mathfrak{g}\rtimes_{Id}\text{Sp}(\mathfrak{g},\omega^+_\epsilon)$
	determined by
	\begin{align*}
	\rho: G &\longrightarrow \mathfrak{g}\rtimes_{Id}\text{Sp}(\mathfrak{g},\omega^+_\epsilon)\\
	\textsf{exp}_G(x) &\longmapsto \rho(\textsf{exp}_G(x))=\left(\sum_{k=1}^\infty \dfrac{1}{k!}(L_x)^{k-1}(x),\sum_{k=0}^\infty \dfrac{1}{k!}(L_x)^k\right),
	\end{align*}
	for all $x\in \mathfrak{g}$. As $\rho$ is an injective homomorphism, we can identify $G$ with the subgroup $\rho(G)$ of $\mathfrak{g}\rtimes_{Id}\text{Sp}(\mathfrak{g},\omega^+_\epsilon)$. Therefore, as $L_b=0$, it follows that $\rho$ determines a nontrivial central one parameter subgroup $H$ of $\rho(G)$ formed by translations which is induced by $t\mapsto \textsf{exp}_G(tb)$ and given by
	$$H=\left\lbrace \rho(\textsf{exp}_G(tb))=(tb,\text{Id}_\mathfrak{g}):\ t\in\mathbb{R}\right\rbrace.$$
	Hence, as $\mathfrak{g}$ is a nilpotent algebra, the existence of such a subgroup $H$ implies that $\mathfrak{g}$ is a nilpotent Lie algebra, and therefore, $G$ is a nilpotent Lie group (see \cite{H}).
\end{proof}
\subsection{Completeness}
We will finish this section giving some comments about the completeness of a left invariant flat affine symplectic connection. It is well known that a left invariant flat affine connection $\nabla$ is geodesically complete if and only if  $\text{tr}(R_x)=0$ for all $x\in\mathfrak{g}$, where $R_x: \mathfrak{g}\to\mathfrak{g}$ is the linear map defined by $R_x(y)=(\nabla_{y^+}x^+)(\epsilon)$ for all $x,y\in\mathfrak{g}$ (see \cite{H}). Assume that $(G,\omega^+,\nabla)$ is a connected flat affine symplectic Lie group and let $(\mathfrak{g},\omega,\cdot)$ be its flat affine symplectic Lie algebra.

Similarly to how it was done in the pseudo-Riemannian case by A. Aubert and A. Medina in \cite{AuM}, we can give sufficient and necessary conditions for determining when a left invariant flat affine symplectic connection is geodesically complete. Denote by  $L_x(y)=R_y(x)=x\cdot y=(\nabla_{x^+}y^+)(\epsilon)$ for all $x,y\in\mathfrak{g}$. Thus, we have
\begin{equation}\label{n12}
\omega(L_x(y),z)+\omega(y,L_x(z))=0,\qquad x,y\in\mathfrak{g}.
\end{equation}
This implies that $L_x\in\mathfrak{sp}(\mathfrak{g},\omega)$ for all $x\in\mathfrak{g}$. Hence, the adjoint map of $L_x$ with respect to $\omega$, which we denote by $L_x^+:\mathfrak{g}\to \mathfrak{g}$, verifies that $L_x^+=-L_x$. On the other hand, if $\mathfrak{g}^\ast$ denotes the dual vector space of $\mathfrak{g}$ and $^tL_x:\mathfrak{g}^\ast\to \mathfrak{g}^\ast$ is the transpose map associated to $L_x$, then the identity \eqref{n12} implies that the linear isomorphism $\omega^\flat:\mathfrak{g}\to\mathfrak{g}^\ast$, defined by $\omega^\flat(x)=\omega(x,\cdot)$, makes the following diagram commutative for all $x\in \mathfrak{g}$
$$\xymatrix{
	\mathfrak{g} \ar[d]_{\omega^\flat}\ar[r]^{L_x^+} & \mathfrak{g} \ar[d]^{\omega^\flat}\\
	\mathfrak{g}^\ast \ar[r]_{^tL_x} & \mathfrak{g}^\ast 
}.$$
So, we have that
$$-\text{tr}(L_x)=\text{tr}(L_x^+)=\text{tr}((\omega^\flat)^{-1} \circ ^tL_x\circ \omega^\flat)=\text{tr}(^tL_x)=\text{tr}(L_x),\qquad x\in\mathfrak{g},$$
and hence $\text{tr}(L_x)=0$ for all $x\in \mathfrak{g}$.

Suppose that $\nabla$ is geodesically complete; so $\text{tr}(R_x)=0$ for all $x\in \mathfrak{g}$. As $\text{ad}_x=L_x-R_x$, we have that 
$$\text{tr}(\text{ad}_x)=\text{tr}(L_x-R_x)=\text{tr}(L_x)-\text{tr}(R_x)=0,\qquad x\in\mathfrak{g},$$
that is, $G$ is a unimodular Lie group. Conversely, if $G$ is a unimodular Lie group, then the identities $\text{ad}_x=L_x-R_x$ and $\text{tr}(L_x)=0$ imply that $\text{tr}(R_x)=\text{tr}(L_x)-\text{tr}(\text{ad}_x)=0$ for all $x\in \mathfrak{g}$. Therefore, $\nabla$ geodesically complete. 

The previous arguments are a different proof for a particular case of the following well known result:
\begin{theorem}\cite{Ba}
	Let $(G,\nabla)$ be a volume-preserving\footnote{$\nabla$ is volume-preserving if there exists a left invariant volume form $\eta$ on $G$ such that $\nabla\eta=0$.} flat affine Lie group. Then $\nabla$ is geodesically complete if and only if $G$ is unimodular. 
\end{theorem}
As every unimodular Lie group admitting a left invariant symplectic form is solvable (see for instance \cite{B}), we get:
\begin{Corollary}\cite{Ba}
	If $(G,\omega^+,\nabla)$ is a connected flat affine symplectic Lie group and $\nabla$ is geodesically complete, then $G$ is solvable.
\end{Corollary}

\section{A double extension of a flat affine symplectic Lie algebra}
By the previous section it is known that there exists a bijective correspondence between simply connected flat affine symplectic Lie groups and flat affine symplectic Lie algebras. The main aim of this section is to give a method for constructing flat affine symplectic Lie algebras using Nijenhuis' cohomology for left symmetric algebras (see \cite{N}). This will be an iterative method which allows us to get flat affine symplectic Lie algebras of dimension $2n+2$ from a flat affine symplectic Lie algebra of dimension $2n$. We will call this construction \emph{a double extension of a flat affine symplectic Lie algebra}. We prove that it is possible to obtain all those flat affine symplectic Lie groups with bi-invariant symplectic connection by means of a double extension starting from $\{0\}$. Moreover, we get all flat affine symplectic Lie groups of dimension $2$ which were found in \cite{An} and give nontrivial examples in every even dimension.

As we said before, we will take the notions about cohomology of left symmetric algebras that we will use from \cite{N}. Let $\mathfrak{g}$ be a left symmetric algebra and $V$ a vector space of finite dimension over $\mathbb{K}=\mathbb{R}$ or $\mathbb{C}$. We say that $V$ has a structure of $\mathfrak{g}$-\emph{bimodule} if there are two bilinear maps $\cdot:\mathfrak{g}\times V\to V$ and $\Box:V\times \mathfrak{g}\to V$ satisfying
\begin{eqnarray}
&  & x\cdot(y\cdot v)-y\cdot(x\cdot v)=[x,y]\cdot v\quad\text{and}\label{bimo1}\\
&  & x\cdot(v\Box y)-(x\cdot v)\Box y=v\Box (xy)-(v\Box x)\Box y, \qquad x,y\in\mathfrak{g},\quad v\in V.\label{bimo2}
\end{eqnarray}
If $C^0(\mathfrak{g},V):=V$ and $C^p(\mathfrak{g},V)$ denotes the set of $p$-linear maps $f:\mathfrak{g}\times\cdots\times\mathfrak{g}\to V$ for $p\in \mathbb{Z}^+$, the \emph{Nijenhuis' differential} $\delta_p:C^p(\mathfrak{g},V)\to C^{p+1}(\mathfrak{g},V)$ is defined as
$$(\delta_p f)(x_0,\cdots,x_p) = \sum_{i=0}^{p-1}(-1)^ix_i\cdot f(x_0,\cdots,\hat{x_i},\cdots,x_p)+\sum_{i=0}^{p-1}(-1)^if(x_0,\cdots,\hat{x_i},\cdots,x_{p-1},x_i)\Box x_p$$
$$-\sum_{i<j<p}(-1)^{i+j+1} f(x_ix_j-x_jx_i,\cdots,\hat{x_i},\cdots,\hat{x_j},\cdots,x_p)-\sum_{i=0}^{p-1}(-1)^if(x_0,\cdots,\hat{x_i},\cdots,x_{p-1},x_ix_p).$$
Using formulas \eqref{bimo1} and \eqref{bimo2} it is simple to check that $\delta_p\circ\delta_{p-1}=0$ (see \cite{N}). As is usual, the spaces of \emph{k-cocycles} (respectively \emph{k-coboundaries}) of the left symmetric algebra $\mathfrak{g}$ with values in $V$ is defined by $Z_{SG}^k(\mathfrak{g},V):=\text{Ker}(\delta_k)$ (respectively $B_{SG}^k(\mathfrak{g},V):=\text{Im}(\delta_{k-1})$ and $B_{SG}^0(\mathfrak{g},V):=\{ 0\}$). Therefore, the $k$-th space of cohomology of the left symmetric algebra $\mathfrak{g}$ with values in $V$ is defined as the quotient space
$$H_{SG}^k(\mathfrak{g},V):=Z_{SG}^k(\mathfrak{g},V)/B_{SG}^k(\mathfrak{g},V).$$
When $V=\mathbb{K}$ has structure of trivial $\mathfrak{g}$-bimodule (i.e. $x\cdot t=t\Box x=0$) we call to $H_{SG}^k(\mathfrak{g},\mathbb{K})$ the $k$-th \emph{space of scalar cohomology} of the left symmetric algebra $\mathfrak{g}$.

If $(\mathfrak{g},\omega,\cdot)$ is a flat affine symplectic Lie algebra and we denote by $L_x(y)=x\cdot y=xy$ for all $x,y\in\mathfrak{g}$, then the map $L:\mathfrak{g} \to \mathfrak{sp}(\mathfrak{g},\omega)$, defined by $x   \mapsto L_x$, is a  well defined linear representation of $\mathfrak{g}$ by $\mathfrak{g}$. That is, the map $L$ is a Lie algebra homomorphism satisfying
\begin{equation}\label{n12.1}
\omega(L_x(y),z)+\omega(y,L_x(z))=0,\qquad x,y,z\in\mathfrak{g}.
\end{equation}
We will denote by $H_L^1(\mathfrak{g},\mathfrak{g})$ the first space of cohomology of the Lie algebra $\mathfrak{g}$ with values in $\mathfrak{g}$ with respect to the representation $L$.
\begin{lemma}\label{F12}
	Let $(\mathfrak{g}, \omega, \cdot)$ be a flat affine symplectic Lie algebra. Then the formula
	$$f(x,y)=\omega(u(x),y),\qquad x,y\in \mathfrak{g},$$
	induces a lineal isomorphism between $H_{SG}^2(\mathfrak{g},\mathbb{K})$ and $H^1_L(\mathfrak{g},\mathfrak{g})$, where $\mathbb{K}$ is seen as a trivial $\mathfrak{g}$-bimodule and $\mathfrak{g}$ has a structure of $\mathfrak{g}$-module determined by the linear representation $L$.
\end{lemma}
\begin{proof}
	It is enough to show that $f$ is a scalar 2-cocycle (respectively scalar 2-coboundary) of left symmetric Lie algebra $\mathfrak{g}$ if and only if $u$ is a $1$-cocycle (respectively $1$-coboundary) of Lie algebra $\mathfrak{g}$. On the one hand, suppose that $f\in Z_{SG}^2(\mathfrak{g},\mathbb{K})$, then we have that $f(xy-yx,z)=f(x,yz)-f(y,xz)$ for all $x,y,z\in\mathfrak{g}$. Therefore
	\begin{eqnarray*}
		\omega(u([x,y]),z) & = & f([x,y],z)\\
		& = &  f(xy-yx,z)\\
		& = & f(x,yz)-f(y,xz)\\
		& = & \omega(u(x),L_y(z))-\omega(u(y),L_x(z))\\
		& = & -\omega(L_y(u(x)),z)+\omega(L_x(u(y)),z)\\
		& = & \omega(L_x(u(y))-L_y(u(x)),z).
	\end{eqnarray*}
	As $\omega$ is nondegenerate, we get that $u([x,y])=L_x(u(y))-L_y(u(x))$ for all $x,y\in\mathfrak{g}$. If $u\in Z^1_L(\mathfrak{g},\mathfrak{g})$, then the previous argument also shows that $f\in Z_{SG}^2(\mathfrak{g},\mathbb{K})$. On the other hand, 
	if $f\in B_{SG}^2(\mathfrak{g},\mathbb{K})$, then there exists a linear map $\varphi:\mathfrak{g}\to \mathbb{K}$ such that $f(x,y)=\varphi(xy)$. As $\omega$ is nondegenerate, there is a unique $x_0\in\mathfrak{g}$ such that $\varphi=\omega(x_0,\cdot)$. Thus
	$$\omega(u(x),y)=f(x,y)=\varphi(xy)=\omega(x_0,\cdot)(xy)=\omega(x_0,xy)=-\omega(L_x(x_0),y).$$
	Therefore, $u(x)=L_x(-x_0)$ for all $x\in\mathfrak{g}$. Now, let $u\in B^1_L(\mathfrak{g},\mathfrak{g})$ and choose $z_0\in\mathfrak{g}$ such that $u(x)=L_x(z_0)$. It follows that $-\omega(z_0,\cdot):\mathfrak{g}\to\mathbb{K}$ is a linear map such that $f(x,y)=-\omega(z_0,\cdot)(xy)$ for all $x,y\in\mathfrak{g}$, that is, $f\in B_{SG}^2(\mathfrak{g},\mathbb{K})$.
\end{proof}
If $(\mathfrak{g},\omega,\cdot)$ is a flat affine symplectic Lie algebra, then by the Lemma \ref{F12} it is simple to check that the map $\theta:H_{L}^1(\mathfrak{g},\mathfrak{g})\to H^1_{ad^*}(\mathfrak{g},\mathfrak{g}^*)$ defined by $[u]\mapsto [\theta(u)]$, where
$$(\theta(u)(x))(y)=\omega(u(x),y)-\omega(u(y),x)=\omega((u+u^*)(x),y),\qquad x,y\in\mathfrak{g},$$
is a well defined linear homomorphism. Here $u^*:\mathfrak{g}\to\mathfrak{g}$ denotes the adjoint map associated to $u$ with respect to $\omega$ and $\text{ad}^\ast:\mathfrak{g}\to\mathfrak{gl}(\mathfrak{g}^\ast)$ the coadjoint representation of $\mathfrak{g}$.

Similarly to how it was done in \cite{AuM} for the pseudo-Riemannian case, we may state and prove the following Lemma for the case we are interested with:
\begin{lemma}\label{F13}
	Suppose that $(\mathfrak{g},\omega,\cdot)$ is a flat affine symplectic Lie algebra. Let $I$ be a bilateral ideal of $(\mathfrak{g},\cdot)$ of dimension $1$. Then
	\begin{enumerate}
		\item The product $\cdot$ in $I$ is null, $I\cdot  I^{\perp_\omega}=0$, and $ I^{\perp_\omega}$ is a left ideal.
		\item $ I^{\perp_\omega}$ is a right ideal if and only if $ I^{\perp_\omega}\cdot I=0$.
		\item If $ I^{\perp_\omega}$ is a bilateral ideal of $(\mathfrak{g},\cdot)$, then the canonical sequences
		\begin{equation}\label{secu1}
		0\longrightarrow I\hookrightarrow I^{\perp_\omega}\longrightarrow I^{\perp_\omega}/ I=B\longrightarrow 0
		\end{equation}
		\begin{equation}\label{secu2}
		0\longrightarrow I^{\perp_\omega}\hookrightarrow \mathfrak{g}\longrightarrow \mathfrak{g}/ I^{\perp_\omega}\longrightarrow 0
		\end{equation}
		\begin{equation}\label{secu3}
		0\longrightarrow I\hookrightarrow \mathfrak{g}\longrightarrow \mathfrak{g}/ I\longrightarrow 0
		\end{equation}
		\begin{equation}\label{secu4}
		0\longrightarrow I^{\perp_\omega}/I\hookrightarrow \mathfrak{g}/I\longrightarrow \mathfrak{g}/ I^{\perp_\omega}\longrightarrow 0,
		\end{equation}
		are sequences of left symmetric algebras. Moreover, the quotient Lie algebra $B=I^{\perp_\omega}/I$ admits a canonical structure of flat affine symplectic algebra and \eqref{secu2}, \eqref{secu4} are split exact sequences of Lie algebras.
	\end{enumerate}
\end{lemma}
\begin{proof}
	Proceeding in order we have: 
	\begin{enumerate}
		\item For $x\in I$, $y\in I^{\perp_\omega}$ and $z\in \mathfrak{g}$, the identity \eqref{n12.1} and the fact that $I$ is 1-dimensional (hence $I\subset I^{\perp_\omega}$) imply that
		$$\omega(x\cdot y,z)=-\omega(y,x\cdot z)=0.$$
		Thus $I\cdot I^{\perp_\omega}=0$. In particular, $I\cdot I=0$ and as
		$$\omega(z\cdot y,x)=-\omega(y,z\cdot x)=0,$$
		we have that $\mathfrak{g}\cdot I^{\perp_\omega}\subset I^{\perp_\omega}$.
		\item For $x\in I$, $y\in I^{\perp_\omega}$ and $z\in \mathfrak{g}$, we get 
		$$\omega(y\cdot z,x)=0 \Leftrightarrow \omega(z,y\cdot x)=0, $$
		this is, $I^{\perp_\omega}\cdot \mathfrak{g}\subset  I^{\perp_\omega}$ if and only if $I^{\perp_\omega}\cdot I=0$. 
		\item Suppose that $I^{\perp_\omega}$ is a bilateral ideal of $(\mathfrak{g},\cdot)$. In this case, the quotient vector space $B=I^{\perp_\omega}/I$ has a structure of Lie algebra given by
		$$[x+I,y+I]=[x,y]+I=(x\cdot y-y\cdot x)+I,\qquad x,y\in I^{\perp_\omega}.$$
		As $\omega$ is a nondegenerate scalar 2-cocycle, it induces on $I^{\perp_\omega}$ a bilinear form of radical $I$. Hence, we have that $\left.\omega\right\vert_{I^{\perp_\omega}\times I^{\perp_\omega}}$ defines, passing to the quotient, a nondegenerate and skew-symmetric bilinear form $\omega'$ on $B=I^{\perp_\omega}/I$ which is also a scalar 2-cocycle of Lie algebra $B$. This symplectic form is given by $\omega'(x+I,y+I)=\omega(x,y)$ for all $x,y\in I^{\perp_\omega}$. If we denote the class of $x\in I^{\perp_\omega}$ module $I$ by $\overline{x}=x+I$, then the left symmetric product
		$$\overline{x}\cdot\overline{y}=(x+I)\cdot (y+I)=x\cdot y+I=\overline{x\cdot y},$$
		satisfies
		$$\omega'(\overline{x}\cdot\overline{y},\overline{z})+\omega'(\overline{y},\overline{x}\cdot\overline{z})=\omega'(\overline{x\cdot y},\overline{z})+\omega'(\overline{y},\overline{x\cdot z})=\omega(x\cdot y,z)+\omega(y,x\cdot z)=0,$$
		for all $x,y,z\in I^{\perp_\omega}$. Therefore, $(B,\omega',\overline{\cdot})$ is a flat affine symplectic Lie algebra. Finally, as $\mathfrak{g}/I^{\perp_\omega}$ has dimension $1$, we have that \eqref{secu2} and \eqref{secu4} are split exact sequences of Lie algebras.
	\end{enumerate}
\end{proof}

The Lie algebra $B$ of Lemma \ref{F13} will be called \emph{flat affine symplectic Lie algebra deduced from} $I^{\perp_\omega}$ \emph{by means of} $I$. Let $(\mathfrak{g},\omega,\cdot)$ be a flat affine symplectic Lie algebra.
\begin{Assumption}
	Assume that both $I$ and $I^{\perp_\omega}$ are bilateral ideals of $(\mathfrak{g},\cdot)$ with $I$ of dimension $1$.
\end{Assumption}
If we put $I=\mathbb{K}e$, let $\mathbb{K}d$ be a 1-dimensional subspace of $\mathfrak{g}$ such that $\omega(e,d)=1$.
We denote by $\overline{B}=(\text{Vect}_{\mathbb{K}}\lbrace e,d\rbrace)^{\perp_\omega}$. Then $\mathfrak{g}$ is identified with the vector space $\mathbb{K}e\oplus \overline{B}\oplus \mathbb{K}d$ and $I^{\perp_\omega}$ with  $\mathbb{K}e\oplus \overline{B}$. The product on $I^{\perp_\omega}$ can be written as
\begin{equation*}
(\lambda e+x)(\mu e+y)=f(x,y)e+x\odot y,\qquad x,y\in \overline{B},\quad\lambda,\mu\in\mathbb{K},
\end{equation*}
where $x\odot y$ is the component of $x\cdot y$ over $\overline{B}$. A direct calculation allows us to verify that the product in $I^{\perp_\omega}$ is left symmetric if and only if the bilinear map $\odot:\overline{B}\times \overline{B}\to\overline{B}$ is a left symmetric product on $\overline{B}$ and $f$ is a scalar 2-cocycle for the left symmetric algebra $(B,\odot)$. On the other hand, it is easy to see that the canonical linear map $\theta:\overline{B}\to B=I^{\perp_\omega}/I$ defined by $x\mapsto \overline{x}=(x+I)$, is an isomorphism of left symmetric algebras (consider $\lambda=\mu=0$ in the previous equation). Therefore, we can identify $\mathfrak{g}$ with the vector space $\mathbb{K}e\oplus B\oplus \mathbb{K}d$ and $I^{\perp_\omega}$ with  $\mathbb{K}e\oplus B$. Denote the left symmetric product of $B$ by $L_x(y)=xy$ for all $x,y\in B$.

According with the cohomology of left symmetric algebras due to A. Nijenhuis, the sequence \eqref{secu1} is described by the cohomology class of a scalar 2-cocycle $f\in Z_{SG}^2(B,\mathbb{K})$ of the left symmetric algebra $B$ (see \cite{N}). Thus, by Lemma \ref{F12}, this sequence is described by the cohomology class of a 1-cocycle $u\in Z_{L}^1(B,B)$ of Lie algebra $B$ with values in $B$ with respect to the linear representation $L:B\to\mathfrak{sp}(B,\omega')$ satisfying $f(x,y)=\omega'(u(x),y)$ for all $x,y\in B$. Thus, the product on $I^{\perp_\omega}$ can be written as
$$(\lambda e+x)(\mu e+y)=f(x,y)e+xy=\omega'(u(x),y)e+xy,\qquad x,y\in B,\quad\lambda,\mu\in\mathbb{K}.$$
As both $I=\mathbb{K}e$ and $I^{\perp_\omega}=\mathbb{K}e\oplus B$ are bilateral ideals of $(\mathfrak{g},\cdot)$, then they are also Lie algebra ideals of $\mathfrak{g}$. Therefore, the Lie bracket of $\mathfrak{g}$ can be expressed  as follows
\begin{eqnarray*}
	&  & [d,e]=\mu e\nonumber\\
	&  & [d,x]=\alpha(x)e+D(x)\\
	&  & [x,y]=\omega'((u+u^*)(x),y)e+[x,y]_B\nonumber,
\end{eqnarray*}
where $x,y\in B$, $\mu\in \mathbb{K}$, $D\in\mathfrak{gl}(B)$, and $\alpha\in B^*$. The product on $\mathfrak{g}=\mathbb{K}e\oplus B\oplus \mathbb{K}d$  satisfies the identity \eqref{n12.1} with $\omega(e,d)=1$ and $\text{Vect}_\mathbb{K}\lbrace e,d\rbrace\perp_\omega B$, if and only if
\begin{eqnarray}\label{productdoble1}
&  & e\cdot x=x\cdot e=e\cdot e=0\nonumber\\
&  & x\cdot y=\omega'(u(x),y)e+xy\nonumber\\
&  & d\cdot x=\omega'(x_0,x)e+p(x)\nonumber\\
&  & x\cdot d=\varphi(x)e+u(x)\\
&  & d\cdot e=\lambda e\nonumber\\
&  & e\cdot d=\gamma e\nonumber\\
&  & d\cdot d=\beta e+x_0-\lambda d\nonumber,
\end{eqnarray}
where $\lambda,\gamma,\beta\in\mathbb{K}$, $x_0\in B$, $p\in\mathfrak{sp}(B,\omega')$, $\varphi\in B^*$ and $u\in Z_{L}^1(B,B)$. The commutator of the product \eqref{productdoble1} agrees with the Lie algebra structure of $\mathfrak{g}$ if and only if, $\gamma=\lambda-\mu$, $\varphi=\omega'(x_0,\cdot)-\alpha$ and $p=D+u$. As $\omega'$ is nondegenerate, there is a unique element $z_0\in B$ such that $\alpha=\omega'(z_0,\cdot)$. Therefore, the Lie bracket of $\mathfrak{g}$ is given by
\begin{eqnarray}\label{bracketdoble2}
&  & [d,e]=\mu e\nonumber\\
&  & [d,x]=\omega'(z_0,x)e+D(x)\\
&  & [x,y]=\omega'((u+u^*)(x),y)e+[x,y]_B\nonumber,\qquad x,y\in B
\end{eqnarray}
and the product \eqref{productdoble1} by
\begin{eqnarray}\label{productdoble2}
&  & e\cdot x=x\cdot e=e\cdot e=0\nonumber\\
&  & x\cdot y=\omega'(u(x),y)e+xy\nonumber\\
&  & d\cdot x=\omega'(x_0,x)e+(D+u)(x)\nonumber\\
&  & x\cdot d=\omega'(x_0-z_0,x)e+u(x)\\
&  & d\cdot e=\lambda e\nonumber\\
&  & e\cdot d=(\lambda-\mu)e\nonumber\\
&  & d\cdot d=\beta e+x_0-\lambda d\nonumber,
\end{eqnarray}
where $\lambda,\mu,\beta\in\mathbb{K}$, $x_0,z_0\in B$, $D\in\mathfrak{gl}(B)$  and $u\in Z_{L}^1(B,B)$ such that $D+u\in\mathfrak{sp}(B,\omega')$.

Finally, the product \eqref{productdoble2} is left symmetric if and only if for all $x,y\in B$ we have that
\begin{enumerate}
	\item $(e\cdot d)\cdot e-(d\cdot e)\cdot e=e\cdot(d\cdot e)-d\cdot(e\cdot e)$. This equality holds if and only if
	\begin{equation}\label{doble1}
	\lambda=\mu\qquad\text{or}\qquad\lambda=\dfrac{\mu}{2}.
	\end{equation}
	\item $(d\cdot x)\cdot d-(x\cdot d)\cdot d=d\cdot(x\cdot d)-x\cdot(d\cdot d)$. This identity is true if and only if
	\begin{align}\label{doble2}
	&[u,D]_{\mathfrak{gl}(B)}=u^2+\lambda u-R_{x_0}\qquad \text{and}\\\label{doble3}
	&D^\ast(x_0-z_0)-2u^\ast(x_0)-2\lambda(x_0-z_0)+(\lambda-\mu)z_0=0,
	\end{align}
	where $D^\ast:\mathfrak{g}\to\mathfrak{g}$ and $u^\ast:\mathfrak{g}\to\mathfrak{g}$ denote the adjoint maps with respect to $\omega'$ associated to $D$ and $u$, respectively.
	\item $(d\cdot x)\cdot y-(x\cdot d)\cdot y=d\cdot(x\cdot y)-x\cdot(d\cdot y)$. This equality is satisfied if and only if
	\begin{align}\label{doble4}
	&\omega'(x_0,xy)=\omega'((u\circ D)(x)-(D\circ u)(x)-u^2(x)-\lambda u(x),y),\quad\text{and}\\\label{doble5}
	&D(x)y+xD(y)-D(xy)=u(xy)-xu(y),\qquad x,y\in B,
	\end{align}
	As $\omega'$ is nondegenerate, it is clear that \eqref{doble4} holds if and only if $[u,D]_{\mathfrak{gl}(B)}=u^2+\lambda u-R_{x_0}$.
	\item $(x\cdot y)\cdot d-(y\cdot x)\cdot d=x\cdot(y\cdot d)-y\cdot(x\cdot d)$. This identity is true if and only if
	\begin{equation}\label{doble6}
	(\lambda-\mu)(u+u^\ast)(x)-2(u\circ u^\ast)(x)=(L_x+R_x^\ast)(x_0-z_0),
	\end{equation}
	where $xy=L_x(y)$ for all $x,y\in B$ and $R_x^\ast:\mathfrak{g}\to\mathfrak{g}$ is the adjoint map with respect to $\omega'$ associated to the linear map $R_x:\mathfrak{g}\to\mathfrak{g}$ given by $y\mapsto R_x(y)=yx$.
\end{enumerate}
Consider the bilinear form $\omega_{u,p}:B\times B\to \mathbb{R}$ defined by
$$\omega_{u,p}:(x,y)\mapsto \omega'((u\circ D)(x)-(D\circ u)(x)-u^2(x)-\lambda u(x),y),\qquad x,y\in B.$$
If we take $\varphi=\omega'(x_0,\cdot)$, then it is easy to see that the formula \eqref{doble4} means that $\omega_{u,p}$ is a scalar 2-coboundary for the left symmetric algebra $B$. On the other hand, the formula \eqref{doble5} tells us that the bilinear map  $\beta_{u,x_0}:B\times B\to B$ defined by $\beta_{u,x_0}(x,y)=u(xy)-xu(y)$ for all $x,y\in\mathfrak{g}$, is the differential of $D$ in the Nijenhuis' cohomology of left symmetric algebra $B$ with values in the canonical $B$-bimodule $B$ (see \cite{N}). Moreover, using again the formula \eqref{doble5} and the fact that $u\in Z_{L}^1(B,B)$ we have that $D:B\to B$ is a Lie algebra derivation. Indeed,
\begin{eqnarray*}
	D([x,y]) & = & D(xy-yx)\\
	& = & D(x)y+xD(y)-u(xy)+xu(y)-D(y)x-yD(x)+u(yx)-yu(x)\\
	& = & [D(x),y]+[x,D(y)]-u([x,y])+L_x(u(y))-L_y(u(x))\\
	& = & [D(x),y]+[x,D(y)],
\end{eqnarray*}

Summing up, we have the following results:
\begin{proposition}\label{F14}
	Let $(\mathfrak{g},\omega,\cdot)$ be a flat affine symplectic Lie algebra. Assume that $I$ and $I^{\perp_\omega}$ are bilateral ideals of $(\mathfrak{g},\cdot)$ with $I$ of dimension $1$. If $B=I^{\perp_\omega}/I$ denotes the flat affine symplectic Lie algebra deduced from $I^{\perp_\omega}$ by means of $I$, then the left symmetric product of $\mathfrak{g}$ is given by \eqref{productdoble2} where $\lambda,\mu,\beta\in\mathbb{K}$,\ $x_0,z_0\in B$,\  $D\in\mathfrak{gl}(B)$ and $u\in Z_{L}^1(B,B)$ such that $D+u\in\mathfrak{sp}(B,\omega')$ satisfying the relations \eqref{doble1}, \eqref{doble3}, $\omega_{u,p}\in B_{SG}^2(B,\mathbb{K})$, \eqref{doble5}, and the identity \eqref{doble6}.
\end{proposition}

Reciprocally, we get a method that allows us to construct flat affine symplectic Lie algebras:
\begin{theorem}\label{F15}
	Let $(B,\omega',\overline{\cdot})$ be a flat affine symplectic Lie algebra. Suppose that $\lambda,\mu\in\mathbb{K}$ verify 
	\eqref{doble1} and $u\in Z_{L}^1(B,B)$ with $D\in\mathfrak{gl}(B)$ are such that $D+u\in\mathfrak{sp}(B,\omega')$ and they satisfy \eqref{doble5}. If $x_0\in B$ verifies that  $\omega_{u,p}\in B_{SG}^2(B,\mathbb{K})$ and there exists $z_0\in B$ such that the equalities \eqref{doble3} and \eqref{doble6} are satisfied, then the vector space $\mathfrak{g}=\mathbb{K}e\oplus B\oplus \mathbb{K}d$ endowed with the left symmetric product \eqref{productdoble2} and the symplectic form $\omega$ which extends $\omega'$ and verifies that $\text{Vect}_\mathbb{K}\lbrace e,d\rbrace$ is a hyperbolic plane orthogonal to $B$, is a flat affine symplectic Lie algebra.
\end{theorem}

To set up terminology we give the following definition:
\begin{Definition}\label{doubleextensionFAS}
	The Lie algebra $(\mathfrak{g},\omega,\cdot)$ obtained from Theorem \ref{F15} is called \emph{the double extension} of the flat affine symplectic Lie algebra $(B,\omega',\overline{\cdot})$ according to $(u,D,x_0,z_0,\beta,\lambda,\mu)$.
\end{Definition}
Recall that by \eqref{doble1} we have that the parameters $\lambda$ and $\mu$ in the double extension of a flat affine symplectic Lie algebra are related to each other as $\lambda=\mu$ or $\lambda=\dfrac{\mu}{2}$. In the case that $\lambda=\mu$, we get the following special consequence of Theorem \ref{F15}.
\begin{Corollary}\label{F17}
	If 	$(G,\omega^+,\nabla)$ is a simply connected flat affine symplectic Lie group whose Lie algebra is obtained as a double extension of a flat affine symplectic Lie algebra according with $(u,D,x_0,z_0,\beta,\mu=\lambda)$, then $G$ is identified with a subgroup of $\mathfrak{g}\rtimes_{Id}\textsf{Sp}(\mathfrak{g},\omega^+_\epsilon)$ containing a nontrivial one parameter subgroup formed by central translations.
\end{Corollary}
\begin{proof}
	Let $(\mathfrak{g},\omega,\cdot)$ be the flat affine symplectic Lie algebra associated to $(G,\omega^+,\nabla)$. If $\mathfrak{g}$ is obtained as a double extension of a flat affine symplectic Lie algebra according with the parameters $(u,D,x_0,z_0,\beta,\lambda,\mu)$ where $\lambda=\mu$, then it decomposes as $\mathfrak{g}=\mathbb{K}e\oplus B\oplus \mathbb{K}d$ and $\omega'=\omega|_B$ with $B\perp_\omega\text{Vect}_\mathbb{K}\lbrace e,d\rbrace$ and  $\omega(e,d)=1$. Therefore, the left symmetric product deduced from $\nabla$ is given by
	\begin{eqnarray*}
		&  & e\cdot x=x\cdot e=e\cdot e=0\\
		&  & x\cdot y=\omega'(u(x),y)e+xy\\
		&  & d\cdot x=\omega'(x_0,x)e+(D+u)(x)\\
		&  & x\cdot d=\omega'(x_0-z_0,x)e+u(x)\\
		&  & d\cdot e=\lambda e\\
		&  & e\cdot d=0\\
		&  & d\cdot d=\beta e+x_0-\lambda d.
	\end{eqnarray*}
	It is clear that $\text{Ker}(L)\neq \lbrace 0\rbrace$ since $L_e=0$. As $G$ is simply connected, using Theorem \ref{CharacterizationLeft} as we did in Proposition \ref{F16} it follows that $\rho$ determines a nontrivial one parameter subgroup $H$ of $\rho(G)\approx G$ formed by central translations which is induced by $t\mapsto \textsf{exp}_G(te)$ and given by
	$$H=\left\lbrace \rho(\textsf{exp}_G(te))=(te,\text{Id}_\mathfrak{g}):\ t\in\mathbb{R}\right\rbrace.$$
\end{proof}
As a consequence of the proof of Proposition \ref{F16}, we have that every flat affine symplectic Lie algebra $(\mathfrak{g},\omega,\cdot)$ such that $\cdot:\mathfrak{g}\times\mathfrak{g}\to\mathfrak{g}$ is an associative product on $\mathfrak{g}$ is obtained as a double extension of flat affine symplectic Lie algebras from $\lbrace 0 \rbrace$. Indeed,
\begin{proposition}\label{F18}
	Let $(G,\omega^+,\nabla)$ be a flat affine symplectic Lie group such that $\nabla$ is bi-invariant. Then the Lie algebra of $G$  is obtained as a double extension of flat affine symplectic Lie algebras starting from $\lbrace 0 \rbrace$.
\end{proposition}
\begin{proof}
	If $(\mathfrak{g},\omega,\cdot)$ is the flat affine symplectic algebra associated to $(G,\omega^+,\nabla)$, then $x\cdot y=L_x(y)=R_y(x)=(\nabla_{x^+}y^+)(\epsilon)$ is an associative product on $\mathfrak{g}$ since $\nabla$ is bi-invariant. Under this assumptions, we know that $\mathfrak{g}$ is a nilpotent algebra and there exists an element $e\neq 0$ in $\mathfrak{g}$ such that $L_e=R_e=0$ (see proof of Proposition \ref{F16}). As consequence $e\in\mathfrak{z}(\mathfrak{g})$ (center of $\mathfrak{g}$) since $[e,x]=L_e(x)-R_e(x)=0$ for all $x\in\mathfrak{g}$. We define by $I=\mathbb{R}e$. It is clear that $I$ is a bilateral ideal of $(\mathfrak{g},\cdot)$ since $e\cdot x=x\cdot e=0$. On the other hand, we have that $I^{\perp_\omega}$ is also a bilateral ideal of $(\mathfrak{g},\cdot)$. Indeed, for all $y\in I^{\perp_\omega}$ and $x\in\mathfrak{g}$ 
	$$\omega(x\cdot y,e)=-\omega(y,x\cdot e)=-\omega(y,0)=0\qquad\text{and}$$
	$$\omega(y\cdot x,e)=-\omega(x,y\cdot e)=-\omega(x,0)=0.$$
	and last claim follows because $\omega$ is nondegenerate. Therefore, the Lie algebra $\mathfrak{g}$ is obtained as a double extension from $(B,\omega',\overline{\cdot})$ where $B=I^{\perp_\omega}/I$. On $\mathfrak{g}=\mathbb{K}e\oplus B\oplus \mathbb{K}d$, the product $\cdot$ is given by the
	formulas \eqref{productdoble2} and we find that the associativity of this product implies the associativity of the product $\overline{\cdot}$ on $B$. Hence, $\mathfrak{g}$ is obtained by a series of double extensions of flat affine symplectic Lie algebras starting from $\lbrace 0\rbrace$.
\end{proof}
By means of the double extension of flat affine symplectic Lie algebras we get every flat affine symplectic Lie algebras in dimension 2. These where initially found by A. Andrada in \cite{An} (see also \cite{MSG}).
\begin{example}\label{Example1C3}
	If in the double extension of flat affine symplectic Lie algebras  we put $B=\lbrace 0 \rbrace$, then we have that $\mathfrak{g}=\text{Vect}_\mathbb{K}\lbrace e,d\rbrace$ with symplectic form $\omega(e,d)=1$ and Lie bracket $[d,e]=\mu e$. In this case, the product \eqref{productdoble2} is given by
	$$\begin{tabular}{l | c  r}
	$\cdot$ & $d$ & $e$ \\
	\hline
	$d$ & $\beta e-\lambda d$ & $\lambda e$\\
	$e$ & $(\lambda-\mu)e $ & $0$\\
	\end{tabular}\quad\text{with}\quad \lambda,\mu,\beta\in\mathbb{K}.
	$$
	\begin{enumerate}
		\item When $\mu=0$, we have that $\mathfrak{g}$ is isomorphic to $\mathbb{R}^2$. Also, as $\mu=0$, by \eqref{doble1} we get that $\lambda=0$. Moreover, taking either $\beta\neq 0$ or $\beta=0$ we get
		$$\begin{tabular}{l | c  r}
		$\cdot$ & $d$ & $e$ \\
		\hline
		$d$ & $\beta e$ & $0$\\
		$e$ & $0$ & $0$\\
		\end{tabular}\qquad\qquad\qquad\qquad \begin{tabular}{l | c  r}
		$\cdot$ & $d$ & $e$ \\
		\hline
		$d$ & $0$ & $0$\\
		$e$ & $0 $ & $0$\\
		\end{tabular}
		$$
		so determining the only two non-isomorphic left invariant flat affine symplectic connections over $\mathbb{R}^2$. When $\beta\neq 0$ all these connections are isomorphic. Moreover, these connections induce on the 2-Torus $\mathbb{T}^2$ structures of flat affine symplectic Lie group. E. Remm and M. Goze found in \cite{RG} all the structures of flat affine Lie group on the abelian Lie group  $\mathbb{R}^2$ and they showed that the only ones of such structures that induce left invariant flat affine structures on $\mathbb{T}^2$ are precisely the structures determined by the left symmetric products above.
		\item On the other hand, when $\mu\neq 0$ we have that $\mathfrak{g}$ is isomorphic to $\mathfrak{aff}(\mathbb{R})$. By \eqref{doble1} we know that $\lambda=\mu$ or $\lambda=\dfrac{\mu}{2}$. So we get 
		$$\begin{tabular}{l | c  r}
		$\cdot$ & $d$ & $e$ \\
		\hline
		$d$ & $\beta e-\mu d$ & $\mu e$\\
		$e$ & $0$ & $0$\\
		\end{tabular}\qquad\qquad\qquad\qquad \begin{tabular}{l | c  r}
		$\cdot$ & $d$ & $e$ \\
		\hline
		$d$ & $\beta e-\dfrac{\mu}{2} d$ & $\dfrac{\mu}{2} e$\\
		$e$ & $-\dfrac{\mu}{2} e $ & $0$\\
		\end{tabular}
		$$
		Each of these families determine a unique connection up to isomorphism.
	\end{enumerate}
\end{example}
\subsection{Nontrivial examples in every even dimension}
We will explain an easy way of getting non-trivial simply connected flat affine symplectic Lie groups of dimension $2n$ from the model space $((\mathbb{R}^{2n-2},+),\omega_0,\nabla^0)$. Here $\omega_0$ is the usual symplectic form and $\nabla^0$ is the usual linear connection on $\mathbb{R}^{2n-2}$. Recall that $\nabla^0$ is a left invariant flat affine connection which is symplectic with respect to $\omega_0$. The abelian flat affine symplectic Lie algebra associated to $(\mathbb{R}^{2n-2},\omega_0,\nabla^0)$ is $(\mathbb{R}^{2n-2},\omega_0,\cdot^0)$ where $x\cdot^0y=0$ for all $x,y\in \mathbb{R}^{2n-2}$ since $\nabla^0_{\partial_i}\partial_j=0$.

Suppose that $u=0$, $x_0=z_0$, and $D\in\mathfrak{sp}(\mathbb{R}^{2n-2},\omega_0)$. The double extension of flat affine symplectic Lie algebras under these conditions is $\mathfrak{g}=\mathbb{R}e\oplus \mathbb{R}^{2n-2}\oplus \mathbb{R}d\approx \mathbb{R}^{2n}$ with nonzero Lie brackets $[d,e]=\mu e$ and $[d,x]=\omega_0(x_0,x)e+D(x)$ for all $x\in \mathbb{R}^{2n-2}$.
It is easy to verify that $\mathfrak{g}$ is isomorphic as Lie algebra to the Lie algebra $\mathbb{R}d\ltimes_{\theta}(\mathbb{R}e\times \mathbb{R}^{2n-2})$ where $\theta:\mathbb{R}d\to \mathfrak{gl}(\mathbb{R}e\times \mathbb{R}^{2n-2})$ is the Lie algebra homomorphism  defined by $\theta(d)(e)=\mu e$ and $\theta(d)(x)=\omega_0(x_0,x)e+D(x)$. Therefore, a simply connected Lie group $G$ with Lie algebra $\mathfrak{g}$ is given by  $G=\mathbb{R}\ltimes_\rho(\mathbb{R}\times\mathbb{R}^{2n-2})$ where $\rho$ is the Lie group homomorphism
$$\rho: \mathbb{R} \to \textsf{GL}(\mathbb{R}\times \mathbb{R}^{2n-2}),\qquad t \mapsto \textsf{Exp}\left\lbrace t\begin{pmatrix}
\mu & [\omega^\flat(x_0)]_\gamma\\
0_{n\times 1} & [D]_\gamma
\end{pmatrix}\right\rbrace$$
where $[\omega^\flat(x_0)]_\gamma$ and $[D]_\gamma$ denote the matrix representations of the linear maps $\omega_0(x_0,\cdot):\mathbb{R}^{2n-2}\to\mathbb{R}$ and $D:\mathbb{R}^{2n-2}\to \mathbb{R}^{2n-2}$, respectively, with respect to a fixed ordered basis $\gamma$ of $\mathbb{R}^{2n-2}$. Let $\omega^+$ be the left invariant symplectic form on $G$ determined by $\omega$, where $\omega$ is the symplectic structure on $\mathfrak{g}$ with $\omega|_{\mathbb{R}^{2n-2}}=\omega_0$, $\text{Vect}_\mathbb{K}\lbrace e,d\rbrace\perp_\omega \mathbb{R}^{2n-2}$, and $\omega(e,d)=1$. Hence, for $\lambda=\mu$ or $\lambda=\dfrac{\mu}{2}$, and $(\lambda-\mu)x_0^+=0$, we have that
\begin{eqnarray}\label{dobleExample1}
&  & \nabla_{e^+}x^+=\nabla_{x^+} e^+=\nabla_{e^+}e^+=\nabla_{x^+}y^+=\nabla_{x^+}d^+=0\nonumber\\
&  & \nabla_{d^+} x^+=\omega_0(x_0,x)e^++D(x)^+\nonumber\\
&  & \nabla_{d^+} e^+=\lambda e^+\nonumber\\
&  & \nabla_{e^+} d^+=(\lambda-\mu)e^+\nonumber\\
&  & \nabla_{d^+} d^+=\beta e^++x_0^+-\lambda d^+\nonumber,
\end{eqnarray}
is a left invariant flat affine symplectic connection on $(\mathbb{R}\ltimes_\rho(\mathbb{R}\times \mathbb{R}^{2n-2}),\omega^+)$.
\begin{remark}
	If $\mu=\lambda$, then $x_0$ can take an arbitrary value. Moreover, if $\mu=0$ then $e\in\mathfrak{z}(\mathfrak{g})$. On the other hand, if $\lambda=\dfrac{\mu}{2}\neq 0$, one necessarily gets that $x_0=0$. The choice of $D\in\mathfrak{sp}(\mathbb{R}^{2n-2},\omega_0)$ is always arbitrary (interesting examples can be obtained if we take $D=J_0$ where $J_0$ is the canonical complex structure on $\mathbb{R}^{2n-2}$). Therefore, these statements allows us to conclude that the above construction produces different nontrivial flat affine symplectic structures on $(\mathbb{R}\ltimes_\rho(\mathbb{R}\times \mathbb{R}^{2n-2}),\omega^+)$. Everyone may convince itself by starting the construction with the canonical abelian flat affine symplectic Lie group $(\mathbb{R}^2,\omega_0,\nabla^0)$ where $\mathfrak{sp}(\mathbb{R}^2,\omega_0)=\mathfrak{sl}(2,\mathbb{R})$.
\end{remark}

\section{Twisted cotangent symplectic Lie group}
Our main purpose in this section is to give a method for constructing left invariant flat affine symplectic connections on the ``twisted'' cotangent symplectic Lie group of a connected flat affine Lie group. We will do it using again Nijenhuis' cohomology for left symmetric algebras. As a first step, we introduce the construction of the classical symplectic cotangent Lie group of a connected flat affine Lie group due to A. Medina and Ph. Revoy (see \cite{MR}). Let $(G,\nabla)$ be a connected flat affine Lie group, $\mathfrak{g}$ its Lie algebra, $\widetilde{G}$ its universal covering Lie group, and $\mathfrak{g}^\ast$ the dual space of $\mathfrak{g}$. We denote by $L:\mathfrak{g}\to \mathfrak{gl}(\mathfrak{g})$ the Lie algebra representation induced by $\nabla$ through the formula $x\cdot y=L_x(y)=(\nabla_{x^+}y^+)(\epsilon)$ for all $x,y\in\mathfrak{g}$. The dual representation associated to $L$ is the Lie algebra homomorphism $L^\ast:\mathfrak{g}\to \mathfrak{gl}(\mathfrak{g}^*)$ defined by $L_x^\ast(\alpha)=-^tL_x(\alpha)=-\alpha\circ L_x$ for all $\alpha\in\mathfrak{g}^\ast$. Passing to exponential, we get a Lie group homomorphism $\Phi:\widetilde{G} \to \textsf{GL}(\mathfrak{g}^*)$ determined by $\displaystyle \textsf{exp}_G(x) \mapsto \sum_{k=0}^{\infty}\dfrac{1}{k!}(L_x^\ast)^k$.
Therefore, the product manifold $T^*\widetilde{G}:=\widetilde{G}\times \mathfrak{g}^*$ has a Lie group structure given by
$$(\sigma,\alpha)\cdot(\tau,\beta)=(\sigma\tau,\Phi(\sigma)(\beta)+\alpha),\qquad\sigma,\tau\in\widetilde{G},\quad\alpha,\beta\in\mathfrak{g}^\ast,$$
this is, the semidirect product of $\widetilde{G}$ with the abelian Lie group $\mathfrak{g}^\ast$ by means of $\Phi$. The Lie group $T^*\widetilde{G}=\widetilde{G}\ltimes_\Phi\mathfrak{g}^*$ is called \emph{classical symplectic cotangent Lie group} associated to the connected flat affine Lie group $(G,\nabla)$. The word ``symplectic" in the name of $T^*\widetilde{G}$ is motivated by the following result:
\begin{proposition}\cite{MR}\label{HessProposition1}
	The Lie algebra of $T^*\widetilde{G}=\widetilde{G}\ltimes_\Phi\mathfrak{g}^*$ is the product vector space $\mathfrak{g}\oplus \mathfrak{g}^\ast$ with Lie bracket
	\begin{equation}\label{Hess1}
	[x+\alpha,y+\beta]=[x,y]+L^*_x(\beta)-L^*_y(\alpha),\qquad x,y\in\mathfrak{g},\quad\alpha,\beta\in\mathfrak{g}^\ast.
	\end{equation}
	Furthermore, $\mathfrak{g}\ltimes_L\mathfrak{g}^\ast$ is a symplectic Lie algebra with the nondegenerate scalar 2-cocycle $\widetilde{\omega}$ defined by 
	\begin{equation}\label{Hess2}
	\widetilde{\omega}(x+\alpha,y+\beta)=\alpha(y)-\beta(x),\qquad x,y\in\mathfrak{g},\quad\alpha,\beta\in\mathfrak{g}^\ast.
	\end{equation}
\end{proposition}
Let $\omega^+$ be the left invariant symplectic form on $T^\ast \widetilde{G}$ induced by $\widetilde{\omega}$. As a first result we get the left invariant flat affine connection on $T^\ast\widetilde{G}$ determined by $\omega^+$; see Remark \ref{inducedconnection}.
\begin{proposition}
	The left invariant flat affine connection $\widetilde{\nabla}$ on $T^\ast\widetilde{G}$ determined by $\omega^+$ is 
	$$\widetilde{\nabla}_{(x+\alpha)^+}(y+\beta)^+=(xy+\textnormal{ad}^*_x(\beta)+\alpha\circ R_y)^+,\qquad x+\alpha,\ y+\beta\in \mathfrak{g}\oplus\mathfrak{g}^\ast,$$
	where $R_y:\mathfrak{g}\to \mathfrak{g}$ is the linear map defined by $x\mapsto R_y(x)=xy=(\nabla_{x^+}y^+)(\epsilon)$.
\end{proposition}
\begin{proof}
	We denote $L_x(y)=R_y(x)=xy=(\nabla_{x^+}y^+)(\epsilon)$ for all $x,y\in \mathfrak{g}$. A. Medina and Ph. Revoy observed in \cite{MR} that the dual vector space $\mathfrak{g}^*$ has a structure of $\mathfrak{g}$-bimodule given by the bilinear maps $\cdot:\mathfrak{g}\times \mathfrak{g}^*\to\mathfrak{g}^*$ and $\Box:\mathfrak{g}^*\times\mathfrak{g} \to  \mathfrak{g}^*$
	which are defined by $(x,\beta) \mapsto x\cdot \beta=\text{ad}^*_x(\beta)$ and $(\alpha,y) \mapsto \alpha\Box y=\alpha\circ R_y$, respectively. Therefore, the product vector space $\mathfrak{g}\times\mathfrak{g}^\ast$
	is a left symmetric algebra with product given by
	\begin{equation}\label{F1P}
	(x+\alpha)\odot(y+\beta)=xy+\text{ad}^*_x(\beta)+\alpha\circ R_y,\qquad x+\alpha,\ y+\beta\in \mathfrak{g}\times\mathfrak{g}^\ast.
	\end{equation}
	For the last claim see \cite{N}. On the one hand, as $\text{ad}_x=L_x-R_x$ for all $x\in\mathfrak{g}$, the commutator of  the left symmetric product given in \eqref{F1P} is
	\begin{eqnarray*}		(x+\alpha)\odot(y+\beta)-(y+\beta)\odot(x+\alpha) & = & (xy+\text{ad}^*_x(\beta)+\alpha\circ R_y)-(yx+\text{ad}^*_y(\alpha)+\beta\circ R_x)\\
		& = & xy-yx+-\beta\circ \text{ad}_x+\alpha\circ R_y+\alpha\circ \text{ad}_y-\beta\circ R_x\\
		& = & [x,y]-\beta\circ L_x+\alpha\circ L_y\\
		& = & [x,y]+L^*_x(\beta)-L^*_y(\alpha).
	\end{eqnarray*}
	On the other hand,
	\begin{eqnarray*}
		\widetilde{\omega}((x+\alpha)\odot(y+\beta),z+\gamma) & = & 	\widetilde{\omega}(xy+\text{ad}^*_x(\beta)+\alpha\circ R_y,z+\gamma)\\
		& = & -(\beta\circ \text{ad}^*_x)(z)+\alpha(zy)-\gamma(xy)\\
		& = & -(\gamma\circ L_x)(y)+(\alpha\circ L_z)(y)-\beta([x,z])\\
		& = & -\widetilde{\omega}(y+\beta,[x,z]+L^*_x(\gamma)-L^*_z(\alpha))\\
		& = & -\widetilde{\omega}(y+\beta,[x+\alpha,z+\gamma]),
	\end{eqnarray*}
	for all $\alpha,\beta,\gamma\in \mathfrak{g}^*$ and $x,y,z\in \mathfrak{g}$. Therefore, the left invariant flat affine connection induced by the $\omega^+$ on $T^\ast\widetilde{G}$ is
	$$\widetilde{\nabla}_{(x+\alpha)^+}(y+\beta)^+=(xy+\text{ad}^*_x(\beta)+\alpha\circ R_y)^+,\qquad x+\alpha,\ y+\beta\in \mathfrak{g}\oplus\mathfrak{g}^\ast.$$
\end{proof}
A foliation $\mathcal{F}$ of a symplectic manifold $(M,\omega)$ is called \emph{Lagrangian} if its leaves are Lagrangian submanifolds of $(M,\omega)$. Two foliations $\mathcal{F}_1$ and $\mathcal{F}_2$ of $M$ are called \emph{transversal} if $TM=T\mathcal{F}_1\oplus \mathcal{F}_2$ where $T\mathcal{F}_j$ denotes the tangent distribution associated to $\mathcal{F}_j$. If $(M,\omega)$ is a symplectic manifold and $\mathcal{F}_1$, $\mathcal{F}_2$ are Lagrangian transversal foliations of $M$, we say that $(M,\omega,\mathcal{F}_1,\mathcal{F}_2)$ is a \emph{bi-Lagrangian manifold}.

The following important theorem is due to H. Hess:
\begin{theorem} \cite{Hs}\label{HessConnection}
	If $(M,\omega,\mathcal{F}_1,\mathcal{F}_2)$ is a bi-Lagrangian manifold, there exists a unique torsion free symplectic linear connection $\nabla$ on $M$ that parallelizes both foliations\footnote{A linear connection $\nabla$ \emph{parallelizes} or \emph{preserves} a regular foliation $\mathcal{F}$ if $\nabla_XY\in \Gamma(T\mathcal{F})$ for all $Y\in \Gamma(T\mathcal{F})$ and $X\in\mathfrak{X}(M)$.
	}. This linear connection satisfies the relation
	$$\nabla_{(X_1+Y_1)}(X_2+Y_2)=(\nabla^\flat_{X_1}X_2+\text{pr}_{\mathcal{F}_1}[Y_1,X_2],\text{pr}_{\mathcal{F}_2}[X_1,Y_2]+\nabla^\flat_{Y_1}Y_2),$$
	where $\nabla^\flat$ is the only symplectic linear connection determined by the identity
	$$\omega(\nabla^\flat_XY,Z)=X\cdot \omega(Y,Z)-\omega(Y,[X,Z]),$$
	for all  $X,Y\in\Gamma(T\mathcal{F}_i)$ and $Z\in\mathfrak{X}(M)$, with $i=1,2$.
\end{theorem}
The torsion free symplectic linear connection given in the previous theorem is called the \emph{Hess connection} or the \emph{canonical connection associated to a bi-Lagrangian manifold}. We will denote this linear connection by $\nabla^H$.
\begin{remark}
	If $(G,\omega^+)$ is a $2n$-dimensional symplectic Lie group and there exist two Lagrangian Lie subalgebras $\mathfrak{g}_1$ and $\mathfrak{g}_2$ of $(\mathfrak{g},\omega^+_\epsilon)$ such that $\mathfrak{g}=\mathfrak{g}_1\oplus \mathfrak{g}_2$, it is simple to show that there exists a symplectic basis $(e_1,\cdots,e_n,\widetilde{e_1},\cdots,\widetilde{e_n})$ of $(\mathfrak{g},\omega^+_\epsilon)$ such that $\mathfrak{g}_1=\text{Vect}_\mathbb{R}(e_1,\cdots,e_n)$ and  $\mathfrak{g}_2=\text{Vect}_\mathbb{R}(\widetilde{e_1},\cdots,\widetilde{e_n})$. Moreover, from Frobenius' Theorem we have that there exist two left invariant transversal Lagrangian foliations $\mathcal{G}_1$ and $\mathcal{G}_2$ of $(G,\omega^+)$ such that $T\mathcal{G}_1$ is generated by $(e_1^+,\cdots,e_n^+)$ and $T\mathcal{G}_2$ by $(\widetilde{e_1}^+,\cdots,\widetilde{e_n}^+)$. A direct computation allows us to prove that the Hess connection on $G$ associated to the bi-Lagrangian Lie group $(G,\omega^+,\mathcal{G}_1,\mathcal{G}_2)$ is left invariant.
\end{remark}
For our case, it is easy to see that $\mathfrak{g}$ and $\mathfrak{g}^\ast$ are contained in $(\mathfrak{g}\oplus \mathfrak{g}^\ast,\widetilde{\omega})$ as Lagrangian Lie subalgebras. If $\mathcal{G}_1$ and $\mathcal{G}_2$ are the left invariant transversal Lagrangian foliations of $(T^\ast \widetilde{G},\omega^+)$ determined by $\mathfrak{g}$ and $\mathfrak{g}^\ast$, respectively, then
\begin{proposition}\label{F19}
	There exists a unique left invariant flat affine symplectic connection $\nabla^H$ on $(T^\ast \widetilde{G},\omega^+)$ that parallelizes both foliations $\mathcal{G}_1$ and $\mathcal{G}_2$. The connection $\nabla^H$ is given by
	\begin{equation}
	\nabla^H_{(x+\alpha)^+}(y+\beta)^+=(xy+L_x^\ast(\beta))^+\qquad x,y\in\mathfrak{g},\quad \alpha,\beta \in\mathfrak{g}^\ast.
	\end{equation}
\end{proposition}

\begin{proof}
	It is clear that if $(e_1,\cdots,e_n)$ is a basis for $\mathfrak{g}$ and $(e_1^\ast,\cdots,e_n^\ast)$ is its dual basis, then $(e_1,\cdots,e_n,e_1^\ast,\cdots,e_n^\ast)$ is a symplectic basis for $(\mathfrak{g}\oplus \mathfrak{g}^\ast,\widetilde{\omega})$ such that 
	$T\mathcal{G}_1=\text{Vect}_\mathbb{K}\lbrace e_1^+,\cdots,e_n^+\rbrace$ and $T\mathcal{G}_2=\text{Vect}_\mathbb{K}\lbrace e_1^*,\cdots,e_n^*\rbrace$. Since $\mathfrak{g}^\ast$ is an abelian Lie algebra, then $e_i^*=(e_i^\ast)^+$ and recall that $e_i=(e_i^+)_\epsilon$ for all $i=1,\cdots,n$.
	As $\omega^+$ is left invariant, we have $\omega^+(e_i^++e_j^\ast,e_k^++e_l^\ast)=\delta_{jk}-\delta_{li}$. It is evident that $\nabla^H$ is the Hess connection, hence it satisfies being the unique torsion free symplectic connection that  parallelizes both foliations $\mathcal{G}_1$ and $\mathcal{G}_2$. If $\displaystyle \nabla^H_{e_i^+}e_j^+=\sum_{\lambda=1}^nc_{ij}^\lambda e_\lambda^+\in \Gamma(T\mathcal{G}_1)$ where $c_{ij}^\lambda$ must be constants, then from Theorem \ref{HessConnection} we get
	\begin{eqnarray*}
		-c_{ij}^k & = & \omega^+\left(\sum_{\lambda=1}^nc_{ij}^\lambda e_\lambda^+,e_k^*\right)\\
		& = & \omega^+(\nabla^H_{e_i^+}e_j^+,e_k^\ast)\\
		& = & e_i^+\cdot \omega^+(e_j^+,e_k^*)-\omega^+(e_j^+,[e_i^+,e_k^*])\\
		& = & -\omega^+(e_j^+,L_{e_i}^*(e_k^*))\\
		& = & -e_k^*((e_i\cdot e_j)^+).
	\end{eqnarray*} 
	Thus $\displaystyle \nabla^H_{e_i^+}e_j^+=\sum_{\lambda=1}^nc_{ij}^\lambda e_\lambda^+=\sum_{\lambda=1}^ne_\lambda^*((e_i\cdot e_j)^+)e_\lambda^+=(e_i\cdot e_j)^+$. As $\nabla^H_{e_i^*}e_j^*\in \Gamma(T\mathcal{G}_2)$, in a similar way as we did above we get that $\nabla^H_{e_i^*}e_j^*=0$. On the other hand, $\nabla^H_{e_i^+}e_j^*\in \Gamma (T\mathcal{G}_2)$. Thus $$\nabla^H_{e_i^+}e_j^*=pr_{\mathcal{G}_2}[e_i^+,e_j^\ast]=pr_{\mathcal{G}_2} (L_{e_i}^\ast(e_j^\ast))=L_{e_i}^\ast(e_j^\ast).$$
	Analogously, as $\nabla_{e_i^*}e_j^+\in \Gamma(T\mathcal{G}_1)$ then $\nabla^H_{e_i^\ast}e_j^+=pr_{\mathcal{G}_1}[e_i^\ast,e_j^+]=pr_{\mathcal{G}_1} (-L_{e_j}^\ast(e_i^\ast))=0$.\\
	Therefore, extending bilinearly, we have that the Hess connection $\nabla^H$ is given by
	$$\nabla^H_{(x+\alpha)^+}(y+\beta)^+=(xy+L_x^\ast(\beta))^+,\qquad x,y\in\mathfrak{g},\quad \alpha,\beta \in\mathfrak{g}^\ast.$$
	Finally, as $\nabla$ is a left invariant flat affine connection on $G$ and $L^\ast:\mathfrak{g}\to\mathfrak{gl}(\mathfrak{g}^\ast)$ is a linear representation, it follows that $(x+\alpha)(y+\beta)=xy+L_x^\ast(\beta)$ is a left symmetric product on $\mathfrak{g}\oplus \mathfrak{g}^\ast$ and hence $\nabla^H$ is also flat.
\end{proof}
Inspired in the construction due to A. Aubert and A. Medina on the twisted cotangent Lie group of a connected flat affine Lie group (in the pseudo-Riemannian case, see \cite{AuM}), we get a more general construction than Proposition \ref{F19}. 

Let $(G,\nabla,\omega^+)$ be a connected flat affine symplectic Lie group and $(\mathfrak{g}, \omega,\cdot)$ its flat affine symplectic Lie algebra.
\begin{Assumption}
	Assume that there exists a Lagrangian bilateral ideal $I$ of $(\mathfrak{g}, \omega, \cdot)$.
\end{Assumption}
By the formula \eqref{n12.1}, we have that
$$\omega(L_a(b),x)=-\omega(b,L_a(x))=0,\qquad a,b\in I,\quad x\in \mathfrak{g}.$$
As $\omega$ is nondegenerate, we get that $L_a(b)=0$ for all $a,b\in I$, this is, $I$ is a left symmetric algebra with null product. Therefore
$$0\longrightarrow I\hookrightarrow\mathfrak{g}\stackrel{\pi}{\hbox to 25pt{\rightarrowfill}}\mathfrak{g}/ I=B\longrightarrow0,$$
is an exact sequence of left symmetric algebras where $I$ has null product. If $s$ is a linear section of $\pi$, the vector space $\mathfrak{g}$ can be identified with $I\oplus s(B)$. In this case, the left symmetric product of $\mathfrak{g}$ can be written as
\begin{equation}\label{cotansym1}
(x+s(a))(y+s(b))=f(s(a),s(b))+xs(b)+s(a)y+s(a)\ast s(b),
\end{equation}
for all $x,y\in I$ and $a,b\in B$, where $f(s(a),s(b))$ and $s(a)\ast s(b)$ denote the components of product $s(a)s(b)$ over $I$ and $s(B)$, respectively. A direct computation shows that \eqref{cotansym1} is a left symmetric product if and only if $\ast$ is a left symmetric product over $S(B)$, the Lagrangian bilateral ideal $I$ has a structure of $s(B)$-bimodule, and $f:s(B)\times s(B)\to I$ is a 2-cocycle of the left symmetric algebra $s(B)$ with values in $I$. Moreover, the linear isomorphism $\left.\pi\right\vert_{s(B)}:s(B)\to B$ defined by $s(a)  \mapsto \pi(s(a))=a$, is also an isomorphism of left symmetric algebras. This follows from \eqref{cotansym1}, since this formula gives
$$\pi(s(a)\ast s(b))=\pi(s(a)s(b)-f(s(a),s(b)))=\pi(s(a))\pi(s(b))=ab,\quad a,b\in B.$$

In what follows we identify the left symmetric algebra $s(B)$ with $B$. Therefore, the left symmetric product in $\mathfrak{g}$ is defined by a structure of $B$-bimodulo over $I$
and a $2$-cocycle of left symmetric algebras $f: B\times B\longrightarrow I$. As $\omega$ is nondegenerate, the bilateral ideal $I$ is a Lagrangian subspace in $(\mathfrak{g},\omega)$, and $\mathfrak{g}=I\oplus B$, we have that the ideal $I$ can be identified with $B^*$ by means of the linear isomorphism $\varphi: I \to B^*$ defined by $x\mapsto\omega(x,\cdot)$.

Now, we determine the structure of $B$-bimodule on $I=B^\ast$ as follows. If $b,b'\in B$ and $\beta\in B^\ast$, as $I$ is Lagrangian we get that
$$\omega(b \cdot \beta,b')+\omega(\beta,b\cdot b')=0 \Leftrightarrow \omega(b \cdot \beta,b')=-\omega(\beta,bb') \Leftrightarrow b\cdot \beta=-^tL_b(\beta)=L^\ast_b(\beta),$$
where $L_b:B\to B$ is defined by $b'\mapsto bb'$ (left symmetric product in $B$). Thus, the left action of $B$ over $B^\ast$ is given by the dual representation $L^\ast$ associated to the linear representation $L:B\to \mathfrak{gl}(B)$ defined by $b\to L_b$. The fact that $L^\ast$ is a Lie algebra homomorphism is equivalent to say that the bilinear map $\cdot:B\times B^\ast\to B^\ast$ defined by $b\cdot \beta= L_b^*(\beta)$, satisfies the identity \eqref{bimo1}. On the other hand, the map $\Phi_{a,b}: B^\ast\to\mathbb{K}$ defined by $\beta\mapsto \omega(\beta\Box a,b)$, allows us to define a product $\circ$ over $B$ by means of the formula
\begin{equation}\label{Important2}
\omega(\beta\Box a,b)=\omega(\beta,a\circ b)\qquad a,b\in B,\quad\beta\in B^\ast.
\end{equation}
As $\omega$ is nondegenerate and $I$ is Lagrangian, we have that both bilinear maps $\Box:B^\ast\times B\to B^\ast$ (defined by $(\beta,b)\to \beta\Box b$) and $\cdot:B\times B^\ast\to B^\ast$ satisfy the identity \eqref{bimo2} if and only if
\begin{eqnarray*}
	&   &\omega(a\cdot(\beta\Box b)-(a\cdot \beta)\Box b,c)=\omega(\beta \Box (ab)-(\beta\Box a)\Box b,c)\\
	& \Leftrightarrow &  -\omega(\beta\Box b,ac)-\omega(a\cdot \beta,b\circ c)=\omega(\beta,(ab)\circ c)-\omega(\beta\Box a,b\circ c)\\
	& \Leftrightarrow & -\omega(\beta,b\circ (ac)-a(b\circ c))=\omega(\beta,(ab)\circ c-a\circ (b\circ c)),
\end{eqnarray*}
that is, the left symmetric product on $B$ and the product $\circ$ must satisfy the relation
\begin{equation}\label{Importan1}
a\circ (b\circ c)+a(b\circ c)=(ab)\circ c+b\circ (ac),\qquad a,b,c\in B.
\end{equation}

Now, we claim that the following map $L':B \to \mathfrak{gl}(B)$ defined by $a\mapsto L'_a$ where $L'_a(b):=a\circ b$, is a Lie algebra homomorphism. To prove this, notice that by the identity \eqref{Important2}, we have that $\beta\Box a=\ ^tL'_a(\beta)$ for all $a\in B$ and $\beta\in B^*$. Therefore, the structure of $B$-bimodule over $I=B^\ast$ is given by 
\begin{equation}\label{Important0}
a\cdot\beta =L^*_a(\beta)\qquad\text{and}\qquad\alpha\Box b=\ ^tL'_b(\alpha).
\end{equation}

Hence, the left symmetric product over $\mathfrak{g}$ given in the formula \eqref{cotansym1} is expressed as
\begin{equation}\label{Important3}
(\alpha+a)(\beta+b)=(\alpha\Box b+a\cdot\beta +f(a,b))+ab=(^tL'_b(\alpha)+L^\ast_a(\beta) +f(a,b))+ab.
\end{equation}

The bracket over $\mathfrak{g}$ given by the commutator of product \eqref{Important3} is as follows
\begin{eqnarray*}
	[(\alpha+a),(\beta+b)] & = & (\alpha+a)(\beta+b)-(\beta+b)(\alpha+a)\\
	& = &  (L^*_a(\beta) -\ ^tL'_a(\beta) -L^*_b(\alpha)+\ ^tL'_b(\alpha)+f(a,b)-f(b,a))\\
	& +& ab-ba\\
	& = & (\theta_a(\beta)-\theta_b(\alpha)+\hat{f}(a,b))+[a,b]_B,
\end{eqnarray*}
where $\theta=L^*-\ ^tL'$ and $\hat{f}:B\times B\to I$ is defined by $\hat{f}(a,b)=f(a,b)-f(b,a)$ for all $a,b\in B$. As $f:B\times B\to I$ is a 2-cocycle of left symmetric algebra $B$ with values in $I$, it is easy to check that $\hat{f}$ is a 2-cocycle of the underlying Lie algebra $B^-$\footnote{The Lie algebra structure over $B^-$ is given by the commutator of the left symmetric product on $B$.} with values in $I$. That the previous bracket is indeed a Lie bracket is equivalent to have that $\theta:B^-\to \mathfrak{gl}(B^*)$ defined by $a\mapsto \theta_a=L^*_a-\ ^tL'_a$, is a Lie algebra homomorphism. As
$$\omega(\theta_a(\beta),c)=\omega(a\cdot\beta-\beta\Box a,c)=\omega(a\cdot\beta,c)-\omega(\beta\Box a,c)=-\omega(\beta,a\circ c+ac)\quad c\in B,$$
the dual representation associated to $\theta$ is given by
\begin{align*} 
\theta^*:B &\longrightarrow \mathfrak{gl}((B^*)^*)\cong \mathfrak{gl}(B)\\
a   &\longmapsto \theta^*_a=-\ ^t\theta_a=-\ ^t(L^*_a-\ ^tL'_a)=L_a+L'_a.
\end{align*}
that is, $\theta^*=L+L'$ is a linear representation of $B$ by $B$. Therefore, $L'$ is also a linear representation of $B$ by $B$, since this is the difference between the linear representations $\theta^*$ and $L$.

Finally, as $I$ is Lagrangian, the left symmetric product on $\mathfrak{g}$ given in the formula \eqref{Important3} is symplectic with respect to $\omega$ if and only if
\begin{eqnarray*}
	&		  &\omega((\alpha+a)(\beta+b),\gamma+c)+\omega(\beta+b,(\alpha+a)(\gamma+c))=0\\
	& \Leftrightarrow &  \omega((^tL'_b(\alpha)-\ ^tL_a(\beta) +f(a,b))+ab,\gamma+c)\\
	& + &\omega(\beta+b,(^tL'_c(\alpha)-\ ^tL_a(\gamma) +f(a,c))+ac)=0\\
	& \Leftrightarrow & \omega(^tL'_b(\alpha)-\ ^tL_a(\beta) +f(a,b),\gamma)+\omega(^tL'_b(\alpha)-\ ^tL_a(\beta) +f(a,b),c)\\
	& + &\omega(ab,\gamma)+\omega(ab,c)+\omega(\beta,^tL'_c(\alpha)-\ ^tL_a(\gamma) +f(a,c))+\omega(\beta,ac)\\
	& + &\omega(b,^tL'_c(\alpha)-\ ^tL_a(\gamma) +f(a,c))+\omega(b,ac)=0\\
	& \Leftrightarrow & (^tL'_b(\alpha)-\ ^tL_a(\beta) +f(a,b))(c)-\gamma(ab)+\beta(ac)\\
	& - &(^tL'_c(\alpha)-\ ^tL_a(\gamma) +f(a,c))(b)+\omega(ab,c)+\omega(b,ac)=0\\
	& \Leftrightarrow & \alpha(b\circ c)-\beta(ac)+f(a,b)(c)-\gamma(ab)+\beta(ac)-\alpha(c\circ b)\\
	& + &\gamma(ab)-f(a,c)(b)=0\\
	& \Leftrightarrow & \alpha(b\circ c-c\circ b)+f(a,b)(c)-f(a,c)(b)=0.
\end{eqnarray*}
for all $a,b,c\in B$ and $\alpha,\beta,\gamma\in B^*$. This is equivalent to having
\begin{equation}\label{Important4}
\alpha(b\circ c-c\circ b)=0\qquad\text{and}\quad f(a,b)(c)-f(a,c)(b)=0,
\end{equation}
for all $a,b,c\in B$ and $\alpha \in B^*$. Therefore, the left symmetric product \eqref{Important3} is symplectic with respect to $\omega$ if and only if the product $\circ$ is commutative and $f$ is a 2-cocycle of the left symmetric algebra $B$ with values in $I$ verifying the formula \eqref{Important4}.

The previous construction shows a short and different approach for obtaining the next result whose reciprocal was proved by X. Ni and C. Bai in \cite{NB}.

\begin{proposition}\label{F20}
	Let  $(\mathfrak{g},\omega,\cdot)$ be a flat affine symplectic Lie algebra which admits a Lagrangian bilateral ideal $I$. Then the canonical exact sequence of left symmetric algebras 
	$$0\longrightarrow I\hookrightarrow\mathfrak{g}\stackrel{\pi}{\hbox to 25pt{\rightarrowfill}}\mathfrak{g}/ I=B\longrightarrow0, $$
	determines on the Lie algebra $B=\mathfrak{g}/ I$ a commutative product $\circ$ that verifies \eqref{Importan1} and endows $B^*$ with a structure of $B$-bimodule of left symmetric algebra given by \eqref{Important0}. Furthermore, the left symmetric algebra $\mathfrak{g}$ is the extension of $B$ by $B^\ast$ according with the 2-cocycle of left symmetric algebra $f\in Z_{SG}^2(B,B^*)$ which satisfies \eqref{Important4}.
	
	Reciprocally, let $(B,\cdot)$ be a left symmetric algebra equipped with a commutative product $\circ$ that satisfies the formula \eqref{Importan1}. Then $B^\ast$ has a structure of $B$-bimodule of left symmetric algebra given by \eqref{Important0}. Moreover, if $f\in Z_{SG}^2(B,B^*)$ is a 2-cocycle of left symmetric algebra that satisfies  \eqref{Important4}, then the vector space $\mathfrak{g}=B^\ast\oplus B$ endowed with the left symmetric product \eqref{Important3} and the symplectic structure given by
	$$\widetilde{\omega}(\alpha+a,\beta+b)=\alpha(b)-\beta(a)\quad a,b\in B,\quad \alpha,\beta\in B^*,$$
	is a flat affine symplectic Lie algebra.
\end{proposition}
To set up terminology we give the following definition:
\begin{Definition}
	The flat affine symplectic Lie algebra of Proposition \ref{F20} is called	\emph{twisted cotangent flat affine symplectic Lie algebra} associated to the left symmetric algebra $B$ according with the commutative product $\circ$ over $B$ and the 2-cocycle of left symmetric algebra $f\in Z_{SG}^2(B,B^\ast)$.
\end{Definition}
It is easy to see that $B^\ast$ is contained on twisted cotangent flat affine symplectic Lie algebra $(B\oplus B^\ast,\widetilde{\omega})$ as an abelian Lagrangian Lie subalgebra.  However, $B$ is contained as Lagrangian Lie subalgebra in $(B\oplus B^\ast,\widetilde{\omega})$ if and only if $f\in Z_{SG}^2(B,B^\ast)$ is symmetric.
\begin{Corollary}\label{F21}
	If $(G,\nabla)$ is a connected flat affine Lie group, then the Hess connection on the classic cotangent symplectic Lie group $(T^\ast \widetilde{G},\omega^+)$ given in the Proposition \ref{F19} is obtained by the previous construction with $f=0$ and null product $\circ$.
\end{Corollary}
\begin{proof}
	If $f=0$ and $\circ$ is the null product on $\mathfrak{g}$, then the left invariant flat affine symplectic connection $\nabla$ over $(T^\ast \widetilde{G},\omega^+)$ determined by \eqref{Important3} is given by
	$$\nabla_{(x+\alpha)^+}(y+\beta)^+=(xy+L_x^\ast(\beta))^+\quad x,y\in\mathfrak{g},\quad \alpha,\beta \in\mathfrak{g}^\ast.$$
	This is precisely the Hess connection of Proposition \ref{F19}, that is, $\nabla=\nabla^H$.
\end{proof}
\begin{remark}
	Suppose that $(G,\nabla)$ is a simply connected flat affine Lie group. In this case the product manifold $T^\ast \widetilde{G}=\widetilde{G}\times \mathfrak{g}^\ast$ is a vector bundle which is isomorphic to the cotangent bundle $T^\ast G$ of $G$. When $f=0$ and the product $\circ$ over $\mathfrak{g}$ is null, $T^\ast G$ is the classic cotangent Lie group of $(G,\nabla)$. Otherwise, when $f$ and $\circ$ are not both null, the simply connected Lie group $H$ with Lie algebra isomorphic to the vector space $\mathfrak{g}\oplus \mathfrak{g}^\ast$ and Lie bracket determined by the commutator of product \eqref{Important3}, that is
	$$[x+\alpha,y+\beta]=[x,y]+\theta_x(\beta)-\theta_y(\beta)+\hat{f}(x,y)$$
	where $\theta=L^*-\ ^tL'$ and $\hat{f}:\mathfrak{g}\times \mathfrak{g}\to \mathfrak{g}^\ast$ is defined by $\hat{f}(x,y)=f(x,y)-f(y,x)$ for all $x,y\in \mathfrak{g}$, will be called the \emph{twisted cotangent symplectic Lie group} of $(G,\nabla)$ according with the commutative product $\circ$ and $f\in Z_{SG}^2(\mathfrak{g},\mathfrak{g}^*)$. By Proposition \ref{F20} and Theorem \ref{CharacterizationLeft} we have that $H$  has a structure of flat affine symplectic Lie group induced by $\widetilde{\omega}$ and the left symmetric product \eqref{Important3}.
\end{remark}
As an immediate consequence of Proposition \ref{F20} we obtain
\begin{theorem}\label{F22}
	Let $G$ be a simply connected Lie group. Then there exists a structure of flat affine symplectic Lie group on $G$ that admits a normal abelian Lagrangian Lie subgroup if and only if, $G$ is isomorphic to the twisted cotangent symplectic Lie group of a connected flat affine Lie group.
\end{theorem}
Assume that $(G,\nabla)$ is a connected flat affine Lie group and $H$ the twisted cotangent symplectic Lie group of $(G,\nabla)$ according with a commutative product $\circ$ and a 2-cocycle  $f\in Z_{SG}^2(\mathfrak{g},\mathfrak{g}^\ast)$ where $ab=(\nabla_{a^+}b^+)(\epsilon)$ for all $a,b\in \mathfrak{g}$. To end the section, we study the completeness of the left invariant flat affine symplectic connection on $(H,\omega^+)$ induced by the left symmetric product \eqref{Important3}. Here $\omega^+$ denotes the left invariant symplectic form on $H$ determined by $\widetilde{\omega}$. Helmstetter showed in \cite{H} that a left invariant flat affine connection $\nabla$ on $G$ is geodesically complete if and only if the ``right multiplication" $R_a:\mathfrak{g}\to\mathfrak{g}$ defined by $R_a(b)=(\nabla_{b^+}a^+)(\epsilon)$, is nilpotent for all $a\in\mathfrak{g}$. Therefore, it is enough to see under what conditions the right product $\widetilde{R}_{a+\alpha}$ is nilpotent for all $a\in\mathfrak{g}$ and $\beta\in\mathfrak{g}^\ast$. The right multiplications on $\mathfrak{g}\times \mathfrak{g}^\ast$ induced by \eqref{Important3} are given by
\begin{eqnarray*}
	&  &\widetilde{R}_a(b)=f(b,a)+ba,\\
	&  &  \widetilde{R}_a(\beta)=\ ^tL'_a(\beta),\\
	&  & \widetilde{R}_\beta(a)=L^*_a(\beta),\quad \text{and}\\
	&  & \widetilde{R}_\alpha(\beta)=0,\quad a,b\in ,\mathfrak{g},\quad \alpha,\beta\in \mathfrak{g}^\ast,
\end{eqnarray*}
where $L'_a:\mathfrak{g}\to\mathfrak{g}$ is defined by $L_a(b)=a\circ b$. Hence, as $\widetilde{R}_\alpha(\beta)=0$ we obtain that $\widetilde{R}_\alpha^2=0$ since $\widetilde{R}_\alpha^2(a)=\widetilde{R}_\alpha(L^*_a(\beta))=0$. On the other hand, if we put $R_a(b)=ab$, then
\begin{eqnarray*}
	&  &\widetilde{R}_a^k(b)=\sum_{j=1}^{k}\ ^t((L'_a)^{j-1})(f(R_a^{k-j}(b),a))+R_a^k(b),\quad\text{and}\\
	&  &  \widetilde{R}_a^k(\beta)=\ ^t(L'_a)^{k}(\beta),\quad k\in \mathbb{N},\quad a,b\in B,\quad \beta\in B^\ast.
\end{eqnarray*}

Thus, we have the following result.
\begin{proposition}\label{F23}
	Let $(G,\nabla)$ be a connected flat affine Lie group. The right product $\widetilde{R}_{a+\alpha}$ associated to the left symmetric product on $\mathfrak{g}\times \mathfrak{g}^\ast$ given by \eqref{Important3} according with the commutative product $\circ$ and the 2-cocycle $f\in Z_{SG}^2(\mathfrak{g},\mathfrak{g}^\ast)$ is nilpotent for all $a\in\mathfrak{g}$ and $\alpha\in \mathfrak{g}^\ast$ if and only if the linear maps $R_a$ and $L'_a$ are nilpotent for all $a\in\mathfrak{g}$. In particular, the Hess connection $\nabla^H$ given in Proposition \ref{F19} is geodesically complete if and only if $\nabla$ is geodesically complete.
\end{proposition}

\section*{Acknowledgments}
I wish to express my sincere gratitude to Omar Saldarriaga and Elizabeth Gasparim for their collaboration and valuable comments. I would also like to thank Alberto Medina for his accompaniment in the Differential Geometry seminar of the Universidad de Antioquia. His paper \cite{AuM} with Anne Aubert was a source of inspiration for the present work. Finally, I want to thank my great friends Sebasti\'an Herrera, Andr\'es Franco, Juli\'an Gonz\'alez, and Adriana Fonce for their company and all the academic discussions that we have had in recent years.

I am grateful for the support given by the Universidad de Antioquia and the Network NT8 from the Office External Activities of Adbus Salam International Center for Theoretical Physics, between years 2017 and 2018.

\end{document}